\documentclass[11pt,letterpaper]{amsart}
\usepackage{latexsym,amsthm,amsmath,amssymb,mathrsfs,mathtools}
\usepackage[T1]{fontenc}
\usepackage[utf8]{inputenc}
\usepackage[mathscr]{eucal}
\usepackage{enumerate}
\usepackage{enumitem}
\usepackage{color}
\usepackage{ esint }
\usepackage[colorlinks	=	true,						
          	 	linkcolor	=	red,
            	urlcolor	=	red,
            	citecolor	=	red]%
            	{hyperref}
\usepackage{cleveref}

\usepackage{todonotes}								

\usepackage{tikz}
\usetikzlibrary{decorations.pathmorphing}	
\usepackage{xcolor}
\usepackage{graphicx}
\usetikzlibrary{calc} 

\newtheorem{theorem}{Theorem}[section]
\newtheorem{lemma}[theorem]{Lemma}
\newtheorem{corollary}[theorem]{Corollary}
\newtheorem{proposition}[theorem]{Proposition}
\theoremstyle{definition}
\newtheorem{remark}[theorem]{Remark}
\newtheorem{definition}[theorem]{Definition}
\newtheorem{example}[theorem]{Example}

\newcommand{\hatc}{\hat{\mathbb{C}}}

\def\diam{\operatorname{diam}}
\def\card{\operatorname{card}}
\def\dist{\operatorname{dist}}

\def\mod{\operatorname{mod}}

\DeclareMathOperator{\Mod}{mod}

\title{Conformal Uniformization of domains bounded by quasitripods} 
\author{Behnam Esmayli, Kai Rajala}
\address{B.\ Esmayli: Department of Mathematical Sciences, P.O. Box 210025, University of Cincinnati, Cincinnati, OH 45221–0025, U.S.A. {\tt esmaylbm@ucmail.uc.edu }} %
\address{K.\ Rajala: Department of Mathematics and Statistics, University of Jyv\"askyl\"a, P.O. Box 35 (MaD), FI-40014, University of Jyv\"askyl\"a, Finland. {\tt kai.i.rajala@jyu.fi}}

\thanks{{B.E.\ was supported by the Research Council of Finland, project numbers 321896 and 323960. This research started when B.E.\ was a postdoctoral researcher at the University of Jyv\"askyl\"a. K.R. was supported by the Research Council of Finland, project number 360505}}
\begin{document}
\begin{abstract}
We prove \emph{Koebe's conjecture} and a version of  \emph{Schramm's cofat uniformization theorem} for domains $\Omega \subset \hatc$ satisfying conditions involving \emph{quasitripods}, i.e., quasisymmetric images of the standard tripod. If the non-point  complementary components of $\Omega$ contain uniform quasitripods with large diameters and satisfy a \emph{packing condition}, then there exists a conformal map $f\colon\Omega \to D$ onto a circle domain $D$. Moreover, $f$ preserves the classes of point-components and non-point components. The packing condition is satisfied if $\Omega$ is \emph{cospread}, i.e., if the complementary components contain uniform quasitripods in all scales. 
\end{abstract}
\maketitle
\section{Introduction}
The Riemann mapping theorem asserts that every simply connected proper subdomain $\Omega \subsetneq \mathbb{C}$ can be conformally mapped onto the unit disk. Koebe provided one of the earliest complete proofs of the theorem and explored its extensions to \emph{multiply connected domains}. 

In \cite{Koe:08}, he conjectured that every domain in the Riemann sphere $\hatc$ is conformally equivalent to a \emph{circle domain}, i.e., a domain whose boundary components are either points or circles. In \cite{Koe:20}, he proved the conjecture for all finitely connected domains. One remarkable feature of Koebe's theorem and his conjecture is that they impose no regularity assumptions on the geometry of the complementary components.

In their breakthrough work, He and Schramm~\cite{HS:93} proved Koebe's conjecture for countably connected domains. Shortly thereafter, Schramm~\cite{Sch:95} 
extended the result to uncountably connected domains whose complementary components are uniformly \emph{fat} (i.e., \emph{Ahlfors $2$-regular}).  

In this paper we prove Koebe's conjecture for domains whose complementary components are \emph{spread} and satisfy a \emph{packing condition}. Let us fix some notation before stating our main results.

Given a domain $G \subset \hatc$, we call a connected component $p$ of $\hatc \setminus G$ \emph{non-trivial} and denote $p \in \mathcal{C}_N(G)$ if $\diam(p \cap \mathbb{C})>0$\footnote{We denote by $\diam(A)$ and $\operatorname{Area}(A)$ the Euclidean diameter and Lebesgue measure of $A \subset \mathbb{C}$, resp. }. Otherwise we call $p$ a \emph{point-component} and denote $p \in \mathcal{C}_P(G)$. Let $\hat{G}=\hatc / \sim$, where 
$$
z \sim w \text{ if either } z=w \in G \text{ or } z,w \in p \text{ for some } p \in \mathcal{C}(G):=\mathcal{C}_N(G) \cup \mathcal{C}_P(G). 
$$
We equip $\hat{G}$ with the quotient. The quotient map is $\pi_G\colon \hatc \to \hat{G}$. By Moore's theorem (see ~\cite[page 3]{Dav:86}), $\hat{G}$ is homeomorphic to $\hatc$.

Every homeomorphism $f\colon G \to G'$ has a unique homeomorphic extension $\hat{f}\colon \hat{G} \to \hat{G'}$. By abuse of notation, we do not make a distinction between $p \in \mathcal{C}(G)$ and $\pi_G(p) \in \hat{G}$. 

Recall that $A \subset \hatc$ is \emph{$\tau$-fat} if for every $z_0 \in A \cap \mathbb{C}$ and every disk $\mathbb{D}(z_0,r)$ that does not contain $A$ we have $\operatorname{Area}(A \cap \mathbb{D}(z_0,r)) \geq \tau r^2$. A domain $\Omega \subset \hatc$ is \emph{cofat} if there is $\tau>0$ so that every $p \in \mathcal{C}_N(\Omega)$ is $\tau$-fat. %

\begin{theorem}[\cite{Sch:95}] \label{renttuihin}
Let $\Omega \subset \hatc$ be a cofat domain. Then there is a conformal map $f\colon\Omega \to D$ onto a circle domain $D$. Moreover, 
$\hat{f}(\mathcal{C}_N(\Omega))=\mathcal{C}_N(D)$ and $\hat{f}(\mathcal{C}_P(\Omega))=\mathcal{C}_P(D)$.  
\end{theorem} 

Theorem \ref{renttuihin} and its proof involving Schramm's \emph{transboundary modulus} have been applied to solve a variety of uniformization problems in Euclidean and metric spaces, see e.g. \cite{Bon11,Mer12,BonMer13,Nta23a}. Towards further applications, it is desirable to find minimal assumptions under which the conclusions of Theorem \ref{renttuihin} hold. In this paper we consider conditions involving \emph{tripods} and \emph{quasisymmetries}. Recall that a homeomorphism $\phi \colon  E \to F$ between subsets of 
$\mathbb{C}$ is \emph{weakly} \emph{$H$-quasisymmetric}, where $H$ is a constant, if for all $z_1,z_2,z_3 \in E$ satisfying  $|z_2-z_1|\leq |z_3-z_1|$, we have
$$
|\phi(z_2)-\phi(z_1)| \leq H|\phi(z_3)-\phi(z_1)| . 
$$ 
The \emph{standard tripod} $T_0 \subset \mathbb{C}$ is the union of segments $[0,e^{i \cdot 2j\pi / 3}]$, $j=0,1,2$. 
\begin{definition}
 We call $T \subset \mathbb{C}$ an \emph{$H$-quasitripod} if there is a weakly $H$-quasisymmetric homeomorphism $\phi \colon T_0 \to T$. 
\end{definition} 
Our main result reads as follows. 
\begin{theorem} \label{tripodkoebe}
Let $\Omega \subset \hatc$ be a domain containing $\infty$. Suppose that there are $H, N\geq 1$ so that 
\begin{itemize}
\item[(i)] every $p \in \mathcal{C}_N(\Omega)$ contains an $H$-quasitripod $T$ with 
$$
\diam(T) \geq \diam(p)/H ,
$$ 
\item[(ii)] $\card \{p \in \mathcal{C}_N(\Omega): \,\diam(p) \geq r, \,  p \cap \mathbb{D}(z_0,r) \neq \emptyset \} \leq N$ for every $z_0 \in \mathbb{C}$ and $r>0$. 
\end{itemize}
Then there exists a conformal homeomorphism $f\colon\Omega \to D$ onto a circle domain $D$. Moreover, 
\begin{equation} \label{claimi2}
\hat{f}(\mathcal{C}_N(\Omega))=\mathcal{C}_N(D) \quad \text{and} \quad  \hat{f}(\mathcal{C}_P(\Omega))=\mathcal{C}_P(D). 
\end{equation}
\end{theorem}

The proof of Theorem \ref{tripodkoebe} relies on transboundary modulus estimates which are significantly more involved than the estimates on cofat domains. The difficulty is that, unlike cofatness, Conditions (i) and (ii) do not imply $\ell^2$-bounds for the diameters of the elements in $\mathcal{C}_N(\Omega)$ (see Example \ref{elltwoexample}). Neither Condition (i) nor Condition (ii) alone guarantees \eqref{claimi2}; see Section \ref{sec:examples}. 

We now introduce a local version of Condition (i) which leads to a M\"obius invariant class of domains that satisfy the conclusions of Theorem \ref{tripodkoebe}. 

\begin{definition}
We say that $A \subset \hatc$ is \emph{$H$-spread} if for every $z_0 \in A \cap \mathbb{C}$ and every $0<r < \diam(A \cap \mathbb{C})$ there is an $H$-quasitripod 
$T \subset A \cap \mathbb{D}(z_0,r)$ with $\diam(T) \geq r/H$. A domain $\Omega \subset \hatc$ is \emph{$H$-cospread} if every 
$p \in \mathcal{C}_N(\Omega)$ is $H$-spread, and \emph{cospread} if $\Omega$ is $H$-cospread for some $H$. 
\end{definition}

The class of cospread domains includes the continuum self-similar trees and uniformly branching trees considered by Bonk-Tran \cite{BonTra21} and Bonk-Meyer \cite{BonMey22}, respectively. 

\begin{proposition} \label{cospreadcon}
Let $\Omega \subset \hatc$ be an $H$-cospread domain. Then Conditions (i) and (ii) in Theorem \ref{tripodkoebe} hold with $H$ and $N=N(H)$. 
Moreover, if $\phi \colon \hatc \to \hatc$ is $\alpha$-quasi-M\"obius then $\phi(\Omega)$ is $H'$-cospread, where $H'$ depends only on $H$ and $\alpha$. 
\end{proposition}
In other words, requiring Condition (i) of Theorem~\ref{tripodkoebe} at all scales implies Condition (ii). 
The class of quasi-M\"obius maps, which we recall in Section \ref{sec:propoproof}, contains all M\"obius transformations. By Theorem \ref{tripodkoebe} and Proposition \ref{cospreadcon}, cospread domains admit conformal maps onto circle domains. 

\begin{corollary}\label{cskoebe} If $\Omega \subset \hatc$ is a cospread domain, then there is a conformal homeomorphism $f\colon\Omega \to D$ onto a circle domain $D$. Moreover, $\hat{f}(\mathcal{C}_N(\Omega))=\mathcal{C}_N(D)$ and 
$\hat{f}(\mathcal{C}_P(\Omega))=\mathcal{C}_P(D)$. 
\end{corollary}

In addition to the previously mentioned results, 
the works of He and Schramm on Koebe’s conjecture and the associated rigidity problem \cite{HeSch94} have inspired several recent developments. Koebe's conjecture has been established for \emph{Gromov-hyperbolic domains} in \cite{KarNta25}, and an approach using \emph{exhaustions} has been explored in \cite{Raj23, NtaRaj23}. Further rigidity results have been developed in \cite{You16, NtaYou20, Nta23b, Raj24}.

We finish the introduction by discussing possible extensions. First, our methods can be adapted to show that if every $p \in \mathcal{C}_N(\Omega)$ in Theorem \ref{tripodkoebe} or Corollary \ref{cskoebe} is the closure of a Jordan domain, then $f$ admits a homeomorphic extension $\bar{f}\colon \overline{\Omega} \to \overline{D}$. 

Another extension concerns versions of the \emph{Brandt-Harrington theorem} for infinitely connected domains, see \cite{Bra80,Har82,Sch:95,Sch96}. Although our results only concern circle domain targets, the estimates below and in the proof of \cite[Theorem 4.2]{Sch:95} suggest that they can be replaced in Theorem \ref{tripodkoebe} and Corollary \ref{cskoebe} with targets $D$ so that if $p \in \mathcal{C}_N(\Omega)$ then $\hat{f}(p) \in \mathcal{C}(D)$ is homothetic to a predetermined fat or spread set $q_p$. 

There are cofat domains that are not cospread and do not satisfy the Quasitripod Condition (i) in Theorem \ref{tripodkoebe}. The proof given below can be modified to show that Condition (i) in Theorem \ref{tripodkoebe} can be replaced with the requirement that ``every $p \in \mathcal{C}_N(\Omega)$ is uniformly fat or satisfies Condition (i)''; see Remark \ref{unifyremark}. It would be interesting to identify natural conditions that define a class of domains encompassing both cofat domains and the domains described in Theorem \ref{tripodkoebe}.

This paper is organized as follows. In Section \ref{sec:modulus} we recall the definition of Schramm's transboundary modulus. In Section \ref{sec:Simon} we state our main modulus estimate, Theorem \ref{thm:mainestimate}, for finitely connected domains satisfying the conditions of Theorem \ref{tripodkoebe}. We proceed to give the proof of Theorem \ref{tripodkoebe}, assuming Theorem \ref{thm:mainestimate} as well as the necessary modulus estimates on circle domains (Proposition \ref{circlemodulus}).

We prove Theorem \ref{tripodkoebe} by approximating $\Omega$ with a decreasing sequence of finitely connected domains $\Omega_j \supset \Omega$ satisfying $\mathcal{C}(\Omega_j) \subset \mathcal{C}_N(\Omega)$. Such an approach is standard and was also used by Schramm \cite{Sch:95}. 
Our new innovation and the main difficulty in the proof of Theorem \ref{tripodkoebe} is establishing Theorem \ref{thm:mainestimate}. The proof is given in Section \ref{sec:proofestimate}. 

Section \ref{sec:cmodulus} contains the proof of Proposition \ref{circlemodulus}, the modulus estimates on circle domains. See e.g. \cite{Sch:95,Bon11,Raj23} for similar estimates. In Section \ref{sec:examples}, we construct examples that illustrate the need for both conditions in Theorem \ref{tripodkoebe}. We prove Proposition \ref{cospreadcon} in Section \ref{sec:propoproof}. 

\section*{Acknowledgments} 
We are grateful to Hrant Hakobyan and Pietro Poggi-Corradini for valuable discussions, and to Chengxi Li and the anonymous referees for their insightful comments, which helped improve the presentation of this work.

\section{Transboundary modulus} \label{sec:modulus}
 Koebe's conjecture concerns conformal equivalence of domains, so it is natural to seek conformally invariant objects. The modulus of path families is one such invariant. The classical definition (see e.g. \cite[Ch. 7]{Hei:01}) concerns only paths \emph{within} the domain, meaning it has no discrete part. 

Schramm \cite{Sch:95} made the ingenious observation that one may allow paths to ``pass through'' the complementary components and, by a natural modification, obtain a modulus that is invariant under conformal homeomorphisms. We now define this modulus and recall its main properties, as it will be our main tool throughout. 

Recall that every (continuous) rectifiable path $\gamma$ defined from a compact interval into $\mathbb{C}$ has an arclength re-parameterization $\gamma_s\colon [0,\ell] \to \mathbb{C}$. Given a Borel function $\rho\colon \mathbb{C} \to [0,\infty]$ and a continuous path $\gamma$, we define 
$$
\int_\gamma \rho\, ds := \int_0^\ell \rho(\gamma_s(t))\, dt 
$$ 
if $\gamma$ is rectifiable, and $\int_\gamma \rho \, ds =\infty$ otherwise. 

For a path defined on an open interval, we define $\int_\gamma \rho\, ds $ to be the supremum of $\int_{\gamma'} \rho\, ds $ over all subpaths 
$\gamma'$ that are defined on compact intervals.

An elementary fact that we will use repeatedly is that if $\gamma\colon [a,b] \to \mathbb C$ is a path and $M>0$ a real number, then
$$
\int_\gamma M\, ds \ge M |\gamma(a)-\gamma(b)|. 
$$

Fix a domain $G \subset \hatc$. The \emph{transboundary modulus} $\Mod(\Gamma)$ of a family $\Gamma$ of paths in $\hat{G}$ is defined by
$$
\Mod(\Gamma)=\inf_{\rho \in X(\Gamma)} \int_{G\cap \mathbb{C}} \rho^2 \, dA + \sum_{p \in \mathcal{C}(G)} \rho(p)^2, 
$$
where $X(\Gamma)$ is the collection of \emph{admissible functions for $\Gamma$}, i.e., Borel functions $\rho\colon \hat{G}\to [0,\infty]$ for which 
$$
1 \leq \int_{\gamma} \rho \, ds +\sum_{p \in \mathcal{C}(G) \cap |\gamma|} \rho(p)  \quad \text{for all } \gamma \in \Gamma.  
$$ 
Here $|\gamma|$ denotes the image of the path $\gamma$ and $\int_{\gamma} \rho \, ds$ is the path integral of the restriction of $\gamma$ to $G$. More precisely, the restriction is 
a countable union of disjoint paths $\gamma_j$, each of which maps onto a component of $|\gamma| \setminus \mathcal{C}(G)$, and we define 
$$
\int_{\gamma} \rho \, ds = \sum_j \int_{\gamma_j} \rho \, ds.  
$$

Technically, Schramm worked with \emph{transboundary extremal length} of $\Gamma$, which equals $\frac{1}{\Mod(\Gamma)}$, and noticed that the proof of the conformal invariance of classical conformal modulus can be generalized to transboundary modulus in a straightforward manner.
\begin{lemma}[\cite{Sch:95}, Lemma 1.1]
\label{modinvariance} 
Suppose that $f\colon G \to G'$ is conformal. Then for every path family $\Gamma$ in $\hat{G}$, we have $\Mod(\Gamma)=\Mod(\hat{f}(\Gamma))$. Here $\hat{f}(\Gamma):=\{\hat{f} \circ \gamma: \, \gamma \in \Gamma\}$. 
\end{lemma}

We will apply the following characterization of path families of non-zero modulus in Section~\ref{sec:examples}. The proof follows directly from the definitions and appropriate scalar multiplications of the admissible functions $\rho$. 

\begin{lemma}
    \label{modlemma}
     A family $\Gamma$ of paths in $\hat{G}$ satisfies $\mod(\Gamma)>0$ if and only if there exists an $M>0$ such that for every admissible function $\rho$ for $\Gamma$ that satisfies
    $$
    \int_{G\cap \mathbb{C}} \rho^2 \, dA + \sum_{p \in \mathcal{C}(G)} \rho(p)^2 =1 ,$$
    we have
    $$
    \int_{\gamma} \rho \, ds +\sum_{p \in \mathcal{C}(G) \cap |\gamma|} \rho(p) \leq M  \quad \text{for some } \gamma \in \Gamma .
    $$
\end{lemma}
We will also apply the following basic properties of the transboundary modulus. The proof can be carried out in the same way as for the classical modulus, see \cite[Proposition 3.1]{HakLi23} for properties (1)-(3) and \cite[Cor. 7.20]{Hei:01} for property (4). Given path families $\Gamma_1$ and $\Gamma_2$ in $\hat{G}$, we say that $\Gamma_1$ \emph{minorizes} $\Gamma_2$ if every $\gamma_2 \in \Gamma_2$ contains a subpath $\gamma_1 \in \Gamma_1$. 
\begin{proposition} \label{transproperties} 
Let $\Gamma_1,\Gamma_2,\ldots$ be path families in $\hat{G}$. The following properties hold: 
\begin{enumerate}
\item If $\Gamma_1 \subset \Gamma_2$, then $\mod(\Gamma_1) \leq \mod(\Gamma_2)$. 
\item If $\Gamma=\cup_j \Gamma_j$, then $\mod(\Gamma) \leq \sum_j \mod(\Gamma_j)$. 
\item If $\Gamma_1$ minorizes $\Gamma_2$, then 
$\mod(\Gamma_1) \geq \mod(\Gamma_2)$.
\item  If $p \in G$ or if $p$ is an isolated point-component of $\hat{G}$, then the modulus of all the paths $\gamma$ in $\hat{G}$ satisfying $p \in |\gamma|$ is zero. 
\end{enumerate}
\end{proposition}

\section{Proof of the main result, Theorem \ref{tripodkoebe}} \label{sec:Simon}
The proof of our main result, Theorem \ref{tripodkoebe}, is based on the following estimate. We denote the open Euclidean disk with center $a \in \mathbb{C}$ and radius $r>0$ by $\mathbb{D}(a,r)$, and its boundary circle by $\mathbb{S}(a,r)$. Moreover, $\mathbb{A}(a,r)$ is the annulus $\mathbb{D}(a,4r)\setminus \overline{\mathbb{D}}(a,r/2)$. 
\begin{theorem} \label{thm:mainestimate} 
Let $\Omega \subset \hatc$ be a finitely connected domain that satisfies Conditions (i) and (ii) in Theorem \ref{tripodkoebe} with some constants $H$ and $N$. Then, there is an $M>0$, depending only on $H$ and $N$, so that if $a \in \mathbb{C}$ and $R>0$, then $\Mod \Gamma \leq M$, where 
\[
\Gamma=\{\text{paths in } \pi_{\Omega}(\overline{\mathbb{A}}(a,R))  \text{ joining } \pi_{\Omega}(\mathbb{S}(a,4R)) \text{ and } \pi_{\Omega}(\mathbb{S}(a,R/2))\}.   
\] 
\end{theorem} 

We postpone the proof of Theorem \ref{thm:mainestimate} until Section \ref{sec:proofestimate}, and first show how it can be applied to prove Theorem \ref{tripodkoebe}. We may assume that $\card \mathcal{C}_N(\Omega)=\infty$, since otherwise Theorem \ref{tripodkoebe} follows from Koebe's theorem, see e.g. \cite[Theorem 9.5]{Bon11}. We enumerate the elements and denote 
$\mathcal{C}_N(\Omega)=\{p_0,p_1,\ldots\}$. It follows directly from the definitions that if Theorem \ref{tripodkoebe} holds for  
$$
\Omega'= \hatc \setminus \overline{\bigcup_{p \in \mathcal{C}_N(\Omega)}  p},  
$$
then the theorem also holds for $\Omega$. Indeed, notice that $\Omega \subset \Omega'$, so, if $f$ maps $\Omega'$ onto a circle domain that satisfies~\eqref{claimi2}, then (the restriction of) $f$ maps $\Omega$ onto a circle domain that satisfies~\eqref{claimi2}, because it maps the point-components of $\Omega$ to point-components. Therefore, we may assume that $\Omega'=\Omega$.

Recall that if $G \subset \hatc$ is a domain and $p \in \mathcal{C}(G)$, we do not make a distinction between $p$ and $\pi_G(p)$. In particular, if $p \subset \mathbb{C}$ then $\diam(\pi_G(p))$ is the Euclidean diameter of $p$. 

Given $k \in \mathbb{N}$, let $\tilde{\Omega}_k=\hatc \setminus (p_0 \cup p_1\cup \cdots \cup p_k)$. By Koebe's theorem there is a conformal 
homeomorphism $g_k \colon \tilde{\Omega}_k \to \tilde{D}_k$ so that $q_{k,\ell}:=\hat{g}_k(p_\ell)$ is a disk (with positive radius) for all $\ell=0,1,\ldots,k$. By postcomposing with a Möbius transformation, we may assume that 
\begin{equation} \label{normali}
q_{k,0}= \hatc \setminus \mathbb{D}(0,1) \quad \text{for all } k=1,2,\ldots. 
\end{equation}

For every $\ell \in \mathbb{N}$, any subsequence of $(q_{k,\ell})_k$ has a further subsequence Hausdorff converging to a limit disk or a point. Therefore we can choose a diagonal subsequence $(g_{k_j})_j$, converging locally uniformly in $\Omega$, so that $q_{k_j,\ell}\to q_\ell$ in the Hausdorff topology for each $\ell$. By normalization \eqref{normali}, the limit map $f$ is non-constant and therefore a conformal homeomorphism from $\Omega$ onto a domain $D$. Each $q_\ell$, $\ell \in \mathbb{N}$, is a disk or a point, and $q_0=\hatc \setminus \mathbb{D}(0,1)$. 

Theorem \ref{tripodkoebe} follows once we have established the following properties: 
\begin{eqnarray}
\label{nina1}
& & \diam(\hat{f}(p))=0  \quad \text{for all } p \in \mathcal{C}_P(\Omega), \\
\label{nina2}
& & q_\ell=\hat{f}(p_\ell) \quad \text{and}\quad \diam(q_\ell)>0  \quad  \text{for all } \ell=0,1,2,\ldots. 
\end{eqnarray} 

\begin{figure}
\begin{tikzpicture}[scale=0.8]
 \filldraw [gray!20] (-.9,-1) circle (2.8cm);
 \draw (-.9,-1) circle (2.8cm);
  
\draw[thick, decorate, decoration={random steps, segment length=5pt, amplitude=1.5pt}]
        (-3,-1) -- (-1.5,-2);      
\draw[thick, decorate, decoration={random steps, segment length=5pt, amplitude=1.5pt}]
        (-.6,-2) -- (-1.5,-2);  
\draw[thick, decorate, decoration={random steps, segment length=5pt, amplitude=1.5pt}]
        (-2,-2.9) -- (-1.5,-2); 
  
\draw[thick, decorate, decoration={random steps, segment length=5pt, amplitude=1.5pt}]
        (-3,0) -- (-1.5,1);      
\draw[thick, decorate, decoration={random steps, segment length=5pt, amplitude=1.5pt}]
        (-.6,1.5) -- (-1.5,1);  
\draw[thick, decorate, decoration={random steps, segment length=5pt, amplitude=1.5pt}]
        (-1.3,-1) -- (-1.5,1); 
  
\draw[thick, decorate, decoration={random steps, segment length=5pt, amplitude=1.5pt}]
        (-1,-1) -- (-.6,-.7);      
\draw[thick, decorate, decoration={random steps, segment length=5pt, amplitude=1.5pt}]
        (-.6,-.7) -- (-.5,-1.2);  
\draw[thick, decorate, decoration={random steps, segment length=5pt, amplitude=1.5pt}]
        (-.6,-.7) -- (-.5,-.4); 

\draw[densely dotted,thick, decorate, decoration={random steps, segment length=5pt, amplitude=1.5pt}]
        (.4,-1) -- (0,-.7);      
\draw[densely dotted,thick, decorate, decoration={random steps, segment length=5pt, amplitude=1.5pt}]
        (0,-.7) -- (.1,-1.2);  
\draw[densely dotted,thick, decorate, decoration={random steps, segment length=5pt, amplitude=1.5pt}]
        (0,-.7) -- (-.1,-.4); 

\draw[densely dotted,thick, decorate, decoration={random steps, segment length=5pt, amplitude=1.5pt}]
        (0.6,-.7) -- (.9,-1);
\draw[densely dotted,thick, decorate, decoration={random steps, segment length=5pt, amplitude=1.5pt}]
        (0.6,-.7) -- (.6,-1.2);  
\draw[densely dotted,thick, decorate, decoration={random steps, segment length=5pt, amplitude=1.5pt}]
        (0.6,-.7) -- (.5,-.4);

\draw[densely dotted,thick, decorate, decoration={random steps, segment length=5pt, amplitude=1.5pt}]
        (1.3,-.1) -- (1.7,-.5);
\draw[densely dotted,thick, decorate, decoration={random steps, segment length=5pt, amplitude=1.5pt}]
        (1.3,-.1) -- (1,.8);  
\draw[densely dotted,thick, decorate, decoration={random steps, segment length=5pt, amplitude=1.5pt}]
        (1.3,-.1) -- (.5,0);
  
\filldraw [gray!20] (7,-1) circle (2.8cm);
 \draw (7,-1) circle (2.8cm);
\fill[white] (7.5,0) circle (.3cm);
  
\draw (7.5,0) circle (.3cm);

\fill[white] (7.5,0) circle (.3cm);
  
\draw (7.5,0) circle (.3cm);

\fill[white] (7,-2) circle (.2cm);
  
\draw (7,-2) circle (.2cm);

\fill[white] (7.5,-1.3) circle (.4cm);
  
\draw (7.5,-1.3) circle (.4cm);
\draw[->] (2.3,-1) -- (3.8,-1);
\filldraw node(p) at (3,-.5) (above) {$\hat{f}_j$};
\filldraw node(p) at (6.9,-2.4) (above) {$\hat{f}_j(p_\ell)$};
\filldraw node(p) at (-.77,-1.2) (above) {$p_\ell$};
\end{tikzpicture}
\caption{$\Omega_j$ is the complement of the union of the solid quasitripods, conformally mapped by $f_j$ onto a circle domain. More components will be included in $C_N(\Omega_k)$ as $k$ increases.}
\end{figure}

We denote $g_{k_j}$ by $f_j$, $\tilde{\Omega}_{k_j}$ by $\Omega_j$, and $\tilde{D}_{k_j}$ by $D_j$. Moreover, 
$$
\text{we fix } \bar{p} \in \mathcal{C}(\Omega) \text{ and any Jordan curve } J \subset \Omega.  
$$
Next let $b \in \Omega \cap N_{R}(\bar{p})$, where $0<R=\dist(\bar{p},J)$ and $N_{\delta}(A)$ is the $\delta$-neighborhood of $A$ in $\mathbb{C}$. Here and in what follows, all distances are Euclidean unless stated otherwise. We choose a point $a \in \partial \bar{p}$ closest to $b$ and denote by $I$ the segment in $\mathbb{C}$ with endpoints $a$ and $b$. Given $j \geq 1$, let 
\begin{eqnarray*} 
\Gamma_j &=& \{\text{paths in } \hat{\Omega}_j\setminus  \{\pi_{\Omega_j}(\bar{p})\} \text{ that join } \pi_{\Omega_j}(J) \text{ and } \pi_{\Omega_j}(I) \}, \\ 
\Lambda_j &=& \{\text{paths in } \hat{\Omega}_j\setminus  \{\pi_{\Omega_j}(\bar{p})\}  \text{ that separate } \pi_{\Omega_j}(J) \text{ and } \pi_{\Omega_j}(\bar{p})\}. 
\end{eqnarray*} 

In summary, here is how the proofs of \eqref{nina1} and \eqref{nina2} proceed. We use Theorem \ref{thm:mainestimate} to prove upper bounds on $\Mod\Gamma_j$ and $\Mod \Lambda_j$. On the other hand, estimates on circle domains $D_j$ provide lower bounds on $ \Mod \hat{f}_j(\Gamma_j) $ and $\Mod \hat{f}_j (\Lambda_j)$. Combined with the conformal invariance of the transboundary modulus, these yield \eqref{nina1} and \eqref{nina2}.

We now state the circle domain estimates; we will prove them later in Section~\ref{sec:cmodulus}. 

\begin{proposition}
\label{circlemodulus} Let $f_j:\Omega_j \to D_j$ be the conformal maps defined above, such that each $D_j$ is a circle domain. The following estimates hold: 
\begin{enumerate}
\item There is a homeomorphism $\varphi_{\bar p}:[0,\infty) \to [0,\infty)$ so that 
$$
\limsup_{j \to \infty} \Mod \hat{f}_j(\Gamma_j) \geq \limsup_{j \to \infty} \varphi_{\bar p}(\dist(f_j(b),\hat{f}_j(\bar p))). 
$$
\item If $\diam(\hat{f}(\bar{p}))=0$ then $\lim_{j \to \infty}\Mod \hat{f}_j (\Lambda_j)=\infty$.  
\end{enumerate} 
\end{proposition}
We now apply Theorem \ref{thm:mainestimate} to establish modulus estimates on $\Gamma_j$, $\Lambda_j$. 
 We first show that  
\begin{equation}
\label{upperbound}
\Mod \Gamma_j \leq \theta_a(|b-a|), 
\end{equation}  
where $\theta_a$ does not depend on $j$ and $\theta_a(\epsilon) \to 0$ as $\epsilon \to 0$. 

To prove \eqref{upperbound}, we notice that every $\gamma \in \Gamma_j$ intersects $\pi_{\Omega_j}(\mathbb{S}(a,R))$ and $\pi_{\Omega_j}(\mathbb{S}(a,|b-a|))$ but avoids $\pi_{\Omega_j}(\bar{p})$. Therefore, by the monotonicity of the transboundary modulus (Proposition \ref{transproperties} (1)), it suffices to show that 
$$
\Mod \Gamma_j(r,R) \leq \theta(r), \quad \theta(r) \to 0 \text{ as } r \to 0, \quad \theta \text{ does not depend on } j,  
$$
where 
\[
\Gamma_j(r,R)=\{\text{paths in } \hat{\Omega}_j\setminus \{\pi_{\Omega_j}(\bar{p})\} \text{ that join } \pi_{\Omega_j}(\mathbb{S}(a,R)) \text{ and } \pi_{\Omega_j}(\mathbb{S}(a,r))\}. 
\] 

We choose a decreasing sequence of radii $R_n$ as follows: Let $R_1:=R/10$. Then, assuming $R_1,\ldots,R_{n-1}$ are defined let 
$$
R_{n}=\frac{R'_n}{10}, 
$$ 
where $R'_n \leq R_{n-1}/2$ is the smallest radius for which some $p \in \mathcal{C}_N(\Omega) \setminus \{\bar{p}\}$ intersects both $\mathbb{S}(a,R_{n-1}/2)$ and $\mathbb{S}(a,R'_n)$. If no $p \in \mathcal{C}_N(\Omega) \setminus \{\bar{p}\}$ intersects $\mathbb{S}(a,R_{n-1}/2)$, we set $R'_n=R_{n-1}/2$. Then $R_{n}$ does not depend on $j$, $R_n \to 0$ as $n \to \infty$, and both the annuli 
$$
\mathbb{A}_n=\mathbb{D}(a,4R_n) \setminus \overline{\mathbb{D}}(a,R_n/2), \quad n=1,2,\ldots, 
$$ 
and their projections $\pi_{\Omega_j}(\mathbb{A}_n)$ under any given $j$ are pairwise disjoint. Let $\Gamma_j(n)$ 
be the family of paths $\gamma$ in 
$\pi_{\Omega_j}(\mathbb{A}_n) \setminus  
\{\pi_{\Omega_j}(\bar{p})\}$ such that 
\[ \gamma \text{ joins }\pi_{\Omega_j}(\mathbb{S}(a,4R_n)) \text{ and } \pi_{\Omega_j}(\mathbb{S}(a,R_{n}/2)).  
\] 
Notice that if $\Omega$ satisfies Conditions (i) and (ii) in Theorem \ref{tripodkoebe} with some $H$ and $N$, then every $\Omega_j$ satisfies the same conditions. Therefore, by Theorem \ref{thm:mainestimate} and the monotonicity of the transboundary modulus (Proposition \ref{transproperties} (1)), we have 
$\Mod \Gamma_j(n) \leq M$, where $M$ does not depend on $j$ or $n$. 

We fix $N \in \mathbb{N}$ and choose for every $1 \leq n \leq N$ an admissible function $\rho_n$ for $\Gamma_j(n)$ such that 
$$
\int_{\Omega_j\cap \mathbb{A}_n} \rho_n^2 \, dA + \sum_{p \in \mathcal{C}(\Omega_j) \cap \pi_{\Omega_j}(\mathbb{A}_n)} \rho_n(p)^2 \leq 2M 
$$ 
and such that $\rho_n(x)=0$ if $x \in \Omega_j \setminus \mathbb{A}_n$ or $x \in \mathcal{C}(\Omega_j) \setminus \pi_{\Omega_j}(\mathbb{A}_n)$. Now $\rho:= \frac{1}{N}\sum_{n=1}^N \rho_n$ is admissible for $\Gamma_j(R_{N+1},R)$. Moreover, since the sets $\pi_{\Omega_j}(\mathbb{A}_n)$ are pairwise disjoint we have 
$$
\int_{\Omega_j} \rho^2 \, dA + \sum_{p \in \mathcal{C}(\Omega_j)} \rho(p)^2 \leq \frac{2MN}{N^2}=\frac{2M}{N} \to 0 
\quad \text{as } N \to \infty. 
$$
Estimate \eqref{upperbound} follows. 

We can now prove \eqref{nina1}: assume $\bar{p}=\{a\} \in \mathcal{C}_P(\Omega)$ and suppose towards a contradiction that $\hat{f}(\bar{p}) \in \mathcal{C}_N(D)$. 
Then there are $c>0$ and a sequence $(b_m)$ of points in $\Omega$ converging to $a$ so that for every $m \in \mathbb{N}$ we have  
\begin{equation} \label{summer} 
\limsup_{j \to \infty}\dist(f_j(b_m),\hat{f}_j(\bar{p})) \geq c >0.  
\end{equation}   
Combining  \eqref{upperbound} and the first part of Proposition \ref{circlemodulus} with Lemma \ref{modinvariance} (conformal invariance of modulus) gives a contradiction, proving \eqref{nina1}. 

Towards \eqref{nina2}, let $\bar{p}=p_\ell$ for some $\ell \in  \mathbb{N} \cup \{0\}$, and let $j_\ell$ be the smallest index for which $p_\ell \in \mathcal{C}_N(\Omega_{j_\ell})$. We claim that 
\begin{equation}\label{lowerbound}
\Mod \Lambda_j \leq M_\ell <\infty \quad \text{for all } j \geq j_\ell, 
\end{equation} 
where $M_\ell$ does not depend on $j$. To start the proof of \eqref{lowerbound}, we fix $c \in \partial p_\ell$ and $d \in J \cap \mathbb{C}$ so that $|c-d|=\dist(p_\ell,J)$, and let 
$\xi$ be the segment with endpoints $c$ and $d$. We cover $\xi$ with $N_1<\infty$ disks $\mathbb{D}(z_n,r)$, where 
$r=\diam(p_\ell)/20$. 

Since every $\lambda \in \Lambda_j$ separates $\pi_{\Omega_j}(\bar p)$ and $\pi_{\Omega_j}(J)$, $\lambda$ has to pass through $\pi_{\Omega_j}(\xi)$ and, consequently, 
through at least one $\pi_{\Omega_j}(\mathbb{D}(z_n,r))$. Furthermore, we have  
$$
\diam(\pi_{\Omega_j}^{-1}(|\lambda|)) \geq \diam(p_\ell), 
$$
which implies that if $\lambda$ passes through $\pi_{\Omega_j}(\mathbb{D}(z_n,r))$ then it also passes through 
$\pi_{\Omega_j}(\mathbb{S}(z_n,8r))$. Therefore, 
\begin{equation} \label{kuppi} 
\Lambda_j \subset \bigcup_{n=1}^{N_1} \Gamma_j(n), 
\end{equation} 
where 
$$
\Gamma_j(n)=\{\text{paths in } \hat{\Omega}_j \text{ joining }  \pi_{\Omega_j}(\mathbb{S}(z_n,8r)) \text{ and } \pi_{\Omega_j}(\mathbb{S}(z_n,r))\}. 
$$ 

By Theorem \ref{thm:mainestimate}, $\Mod \Gamma_j(n) \leq M$ 
for every $1 \leq n \leq N_1$. Thus, by \eqref{kuppi} and the monotonicity and subadditivity of the transboundary modulus (Proposition \ref{transproperties} (1) and (2)), we have 
$$
\Mod \Lambda_j \leq \sum_{n=1}^{N_1} \Mod \Gamma_j(n) \leq MN_1, 
$$
which proves \eqref{lowerbound}. 

We can now prove \eqref{nina2}. The proof of the first part is similar to the proof of \eqref{nina1}. We have $q_\ell \subset \hat{f}(p_\ell)$ by Carath\'eodory's kernel convergence theorem; see \cite[Theorem V.5.1, p.~228]{Gol:69}. %
Suppose towards a contradiction that 
$q_\ell \subsetneq \hat{f}(p_\ell)$. Then there are $c>0$ and a sequence $(b_m)$ in $\Omega$ so that $\operatorname{dist}(b_m,p_\ell) \to 0$ as $m\to \infty$ and \eqref{summer} holds with $\bar{p}=p_\ell$. Combining \eqref{upperbound} and the first part of Proposition \ref{circlemodulus} with Lemma \ref{modinvariance} (conformal invariance of modulus) gives a contradiction. For the second part of \eqref{nina2} it suffices to combine \eqref{lowerbound} and the second part of Proposition \ref{circlemodulus} with Lemma \ref{modinvariance}. 

We have proved that Theorem \ref{tripodkoebe} follows from Theorem \ref{thm:mainestimate} and Proposition \ref{circlemodulus}. 

%

\section{Proof of Theorem \ref{thm:mainestimate} } \label{sec:proofestimate}
In this section we assume that $\Omega \subset \hatc$ is as in Theorem~\ref{thm:mainestimate}: a \emph{finitely connected} domain satisfying Conditions (i) and (ii) in Theorem \ref{tripodkoebe} with some $H$ and $N$. To prove Theorem \ref{thm:mainestimate}, we must find a uniform bound for $\Mod \Gamma$, where (see Figure~\ref{fig:thm3.1}). 
\[
\Gamma=\{\text{paths in } \pi_{\Omega}(\overline{\mathbb{A}}(a,R))  \text{ joining } \pi_{\Omega}(\mathbb{S}(a,4R)) \text{ and } \pi_{\Omega}(\mathbb{S}(a,R/2))\}.    
\] 

\begin{figure}[h]
\begin{tikzpicture}[scale=1]
\filldraw [gray!20] (0,0) circle (3cm);
\filldraw [white] (0,0) circle (1cm);
\draw[thick, decorate, decoration={random steps, segment length=5pt, amplitude=1.5pt}]
        (-.6,0) -- (-1.3,0.1);
\draw[thick, decorate, decoration={random steps, segment length=5pt, amplitude=1.5pt}]
        (-1.6,-0.5) -- (-1.3,0.1);
\draw[thick, decorate, decoration={random steps, segment length=5pt, amplitude=1.5pt}]
        (-1.7,0.6) -- (-1.3,0.1);

\draw[thick, decorate, decoration={random steps, segment length=5pt, amplitude=1.5pt}]
        (-.4,1.5) -- (-0.2,1.2);
\draw[thick, decorate, decoration={random steps, segment length=5pt, amplitude=1.5pt}]
        (+.1,1.5) -- (-0.2,1.2);
\draw[thick, decorate, decoration={random steps, segment length=5pt, amplitude=1.5pt}]
        (-.2,1) -- (-0.2,1.2);

\draw[thick, decorate, decoration={random steps, segment length=5pt, amplitude=1.5pt}]
        (-.4,1.5+.3) -- (-0.2,1.2+.3);
\draw[thick, decorate, decoration={random steps, segment length=5pt, amplitude=1.5pt}]
        (+.1,1.5+.3) -- (-0.2,1.2+.3);
\draw[thick, decorate, decoration={random steps, segment length=5pt, amplitude=1.5pt}]
        (-.2,1+.3) -- (-0.2,1.2+.3);
\draw[thick, decorate, decoration={random steps, segment length=5pt, amplitude=1.5pt}]
        (-.4,1.5+.6) -- (-0.2,1.2+.6);
\draw[thick, decorate, decoration={random steps, segment length=5pt, amplitude=1.5pt}]
        (+.1,1.5+.6) -- (-0.2,1.2+.6);
\draw[thick, decorate, decoration={random steps, segment length=5pt, amplitude=1.5pt}]
        (-.2,1+.6) -- (-0.2,1.2+.6);
\draw[thick, decorate, decoration={random steps, segment length=5pt, amplitude=1.5pt}]
        (-.4,1.5+.9) -- (-0.2,1.2+.9);
\draw[thick, decorate, decoration={random steps, segment length=5pt, amplitude=1.5pt}]
        (+.1,1.5+.9) -- (-0.2,1.2+.9);
\draw[thick, decorate, decoration={random steps, segment length=5pt, amplitude=1.5pt}]
        (-.2,1+.9) -- (-0.2,1.2+.9);
\draw[thick, decorate, decoration={random steps, segment length=5pt, amplitude=1.5pt}]
        (-.4,1.5+1.2) -- (-0.2,1.2+1.2);
\draw[thick, decorate, decoration={random steps, segment length=5pt, amplitude=1.5pt}]
        (+.1,1.5+1.2) -- (-0.2,1.2+1.2);
\draw[thick, decorate, decoration={random steps, segment length=5pt, amplitude=1.5pt}]
        (-.2,1+1.2) -- (-0.2,1.2+1.2);

\draw[thick, decorate, decoration={random steps, segment length=5pt, amplitude=1.5pt}]
        (-2,-3) -- (-1,-1.2) -- (-0.2,-.3);
\draw[thick, decorate, decoration={random steps, segment length=5pt, amplitude=1.5pt}]
        (3,-2.8) -- (1,-1) -- (-0.2,-.3);
\draw[thick, decorate, decoration={random steps, segment length=5pt, amplitude=1.5pt}]
        (-.9,2.2) --(-.8,1) -- (-0.2,-.3);

\draw[thick, decorate, decoration={random steps, segment length=5pt, amplitude=1.5pt}]
        (+.6,1) -- (+1.3,.8);
\draw[thick, decorate, decoration={random steps, segment length=5pt, amplitude=1.5pt}]
        (+1.7,-0.1) -- (+1.3,0.8);
\draw[thick, decorate, decoration={random steps, segment length=5pt, amplitude=1.5pt}]
        (+1.7,1.6) -- (+1.3,0.8);
\end{tikzpicture}
\caption{The annulus $\mathbb{A}(a,R)$ and some of the complementary components of $\Omega$, as in Theorem~\ref{thm:mainestimate}. Paths can use complementary components as shortcuts to join the two boundary circles of the annulus.}
\label{fig:thm3.1}
\end{figure}

We begin by discussing the core ideas of the proof of Theorem \ref{thm:mainestimate}. 
We need to find a Borel function $\rho\colon \hat{\Omega}\to [0,\infty]$ such that
\begin{equation} 
\label{qina}
1 \leq \int_{\gamma} \rho \, ds +\sum_{p \in \mathcal{C}(\Omega) \cap |\gamma|} \rho(p)  \quad \text{for all } \gamma \in \Gamma  
\end{equation}
while maintaining a uniform bound, independent of the center of the annulus and the number of complementary components of $\Omega$, on
\begin{equation}
\label{queen}
\int_{\Omega \cap \mathbb{C}} \rho^2 \, dA + \sum_{p \in \mathcal{C}(\Omega)} \rho(p)^2.
\end{equation}

First, let $\tilde{\rho}(x)=R^{-1}$ when $x \in \Omega \cap \mathbb{A}(a,R)$, and $\tilde{\rho}(x)=0$ elsewhere. Then $\tilde{\rho}$   
is admissible for the subfamily of $\Gamma$ consisting of the paths which stay within the domain $\Omega$. Moreover, the energy \eqref{queen} of $\tilde{\rho}$ is bounded from above by $16\pi$.  
It is therefore natural to define $\rho=\tilde{\rho}$ on $\Omega$.  

The first attempt towards completing the definition of $\rho$ is setting $\rho(p)=\diam(p)/R$ for all $p \in \mathcal{C}(\Omega)$. The triangle inequality shows that such a $\rho$ satisfies the admissibility condition \eqref{qina}. Moreover, if $\Omega$ is a circle domain or a cofat domain, then the energy \eqref{queen} is uniformly bounded from above.  

However, our assumptions are not restrictive enough to guarantee uniform bounds for the energy \eqref{queen} of such a $\rho$, see Example \ref{elltwoexample}. 
This is a serious obstacle, which we overcome by finding an intricate definition of $\rho$ on $\mathcal{C}(\Omega)$. 

To this end, the first step is to notice that the packing Condition (ii) in Theorem \ref{tripodkoebe} gives a uniform bound for the number of complementary components intersecting $\mathbb{A}(a,R)$ whose diameters are larger than or comparable to $R$. Thus, we can assign the value $\rho(p)=1$ for each such component $p$.

To complete the definition of $\rho$, we apply the quasitripod Condition (i) in Theorem \ref{tripodkoebe} to show that the elements of a substantial subset $B$ of $\mathcal{C}(\Omega)$ are ``overshadowed'' by larger components which are not in $B$. We set  $\rho(p)=0$ on $B$ and $\rho(p)=\diam(p)/R$ for the remaining components $p$. 

The most technical part of the proof of Theorem \ref{thm:mainestimate} is proving that our definition yields an admissible function $\rho$ for $\Gamma$. The proof goes roughly as follows: given a path $\gamma \in \Gamma$ passing through some elements of $B$, we need to compensate for the fact that $\rho=0$ on $B$. We apply the following strategy: let $p \in B$ be a component which is overshadowed by a component $p'$. The quasitripod condition implies that there are two options: 
\begin{enumerate} 
\item If $\gamma$ passes through $p'$, the weight $\rho(p')$ is sufficient to compensate for $\rho(p)=0$. 
\item If $\gamma$ passes through $p$ but not through $p'$, then it must ``go around'' $p'$ and pick up ``extra weight'' which is sufficient to compensate for $\rho(p)=0$; see Proposition \ref{blockingprop}. 
\end{enumerate}

The main estimate in the proof of the admissibility of $\rho$ is Proposition \ref{goodpathsprop}, whose proof occupies the last subsections of this section. 

\subsection{Costs of detours around quasitripods}
We now start the proof of Theorem \ref{thm:mainestimate}. We lose no generality by assuming that $\mathcal{C}(\Omega)=\mathcal{C}_N(\Omega)$. Indeed, since $\Omega$ is finitely connected, the point-components $p \in \mathcal{C}_P(\Omega)$ are isolated and we can apply Proposition \ref{transproperties} (4). 

Our main application of the Quasitripod Condition (i) in Theorem \ref{tripodkoebe} is the following proposition, which states that there is a relatively large neighborhood near every $p \in \mathcal{C}(\Omega)$ which is ``overshadowed'' by $p$.

Given $p \in \mathcal{C}(\Omega)$ and $0<\tau<1/4$ we will repeatedly use the shorthand 
\begin{equation} \label{baii}
r_p=r_p(\tau)=\tau \diam(p)>0.
\end{equation} 
Let $a_p,b_p \in \mathbb{C}$ (soon to be specified) and assume that $\overline{\mathbb{D}}(a_p,4\tau r_p) \subset \mathbb{D}(b_p, r_p)$. Let $\Gamma(a_p,b_p,\tau)$ be the family of paths 
\begin{equation} \label{jolly}
\alpha:I \to \hat{\Omega}, \quad I=[s_1,t_1]=[s_1(\alpha),t_1(\alpha)], 
\end{equation}
for which there are $s_1 < s_2 \leq t_2 < t_1$ with the following properties:  
\begin{itemize}
\item[(i)] $\alpha(s_2) \cup \alpha(t_2) \subset \mathbb{D}(a_p,4\tau r_p)$, 
\item[(ii)] $\alpha(t) \cap \mathbb{S}(b_p,r_p) \neq \emptyset$ for $t=s_1$ and $t=t_1$, 
\item[(iii)]  $\diam (\alpha(t)) \le \tau r_p$ for $t=s_1$ and $ t=t_1$, 
\item[(iv)] $\alpha(t) \subset \mathbb{D}(b_p,r_p)$ for all $s_1<t<s_2$ and $t_2<t<t_1$. 
\end{itemize} 
Observe that the subpaths of $\alpha$ on $[s_1,s_2]$ and then on $[t_2,t_1]$ each join $\pi_\Omega (\mathbb{S}(b_p,r_p))$ to $\pi_\Omega(\mathbb{D}(a_p,4\tau r_p))$ within $\pi_\Omega(\overline{\mathbb{D}}(b_p,r_p))$. 

Recall that we assume that every $p \in \mathcal{C}(\Omega)$ contains an $H$-quasitripod $T$ with $\diam(T) \geq \diam(p)/H$. 

\begin{proposition}\label{blockingprop}
There is a number $0<\tau <\frac{1}{1000}$, depending only on $H$, so that the following holds: for every $p \in \mathcal{C}(\Omega)$ there exist $b_p \in p$ and $a_p \in \mathbb{C}$, such that   
\begin{equation}
\label{indika}
\mathbb{D}(a_p,4 \tau r_p) \subset \mathbb{D}(b_p,\tau^{1/2} r_p), 
\end{equation}
and such that for every $\alpha \in \Gamma(a_p,b_p,\tau)$ for which $p \notin |\alpha|$ we have 
\begin{equation} \label{dessed} 
\dist(\alpha(s_1),\alpha(t_1)) \leq \dist(\alpha(s_1),\alpha(s_{2}))+\dist(\alpha(t_1),\alpha(t_2))  - \frac{1}{10} r_p. 
\end{equation} 
Here $s_2$ and $t_2$ are the numbers defined after \eqref{jolly}. 
\end{proposition}

Proposition \ref{blockingprop} implies that if $\alpha \in \Gamma(a_p,b_p,\tau)$ does not pass through $p$ then it is uniformly far from being a ``geodesic'', see Figure~\ref{cost}. Notice that the right side of \eqref{dessed} does not include the term $\dist(\alpha(s_2),\alpha(t_2))$. 

\begin{figure}
\begin{tikzpicture}[scale=1.5]
\draw [thick] (1.8,1.8)--(0,0); 
\draw [thick] (1,-2.6)--(0,0);
\draw [thick] (-3,-.8)--(0,0); 

\draw[thick, decorate, decoration={random steps, segment length=5pt, amplitude=1.5pt}]
        (-1,-1) -- (-.6,-.7);      
\draw[thick, decorate, decoration={random steps, segment length=5pt, amplitude=1.5pt}]
        (-.6,-.7) -- (-.5,-1.2);  
\draw[thick, decorate, decoration={random steps, segment length=5pt, amplitude=1.5pt}]
        (-.6,-.7) -- (-.5,-.4);

\draw[thick, decorate, decoration={random steps, segment length=5pt, amplitude=1.5pt}]
        (1,0) -- (1,0.5);
\draw[thick, decorate, decoration={random steps, segment length=5pt, amplitude=1.5pt}]
        (1,0) -- (.5,-.4);
\draw[thick, decorate, decoration={random steps, segment length=5pt, amplitude=1.5pt}]
        (1,0) -- (1.4,-.2);
        
\draw[thick, decorate, decoration={random steps, segment length=5pt, amplitude=1.5pt}]
        (0,-1.96) -- (0.1,-1.8);
\draw[thick, decorate, decoration={random steps, segment length=5pt, amplitude=1.5pt}]
        (0,-1.96) -- (-.1,-1.92);
\draw[thick, decorate, decoration={random steps, segment length=5pt, amplitude=1.5pt}]
        (0,-1.96) -- (0.05,-2.1);  
        
\draw [thick,densely dotted] plot [smooth ] coordinates {(-1.65,-1.1) (-.95,-.95)};
\draw [thick,densely dotted] plot [smooth ] coordinates {(-.5,-.4) (-0.1,-.17)}; 
\draw [thick,densely dotted] plot [smooth ] coordinates {(0.-.06,-.22) (0.1,-1.8)};

\draw[thin,dashed] (0,0) circle (2cm);
\draw[thin,dashed] (-0.1,-.2) circle (.1cm);
\node at (1,1.2) {$p$};
\node at (-1.4,-.9) {$_{\alpha}$};
\node at (.1,-1) {$_{\alpha}$};
\node at (0.1,-2.24){$_{\alpha(t_1)}$};
\node at (-1.94,-1.1){$_{\alpha(s_1)}$};
\end{tikzpicture}
\caption{The dotted path shows a sample $\alpha \in \Gamma(a_p,b_p,\tau)$. It has two subpaths that each join $\mathbb{S}(b_p,r_p)$ to $\mathbb{D}(a_p,4\tau r_p)$.}
\label{cost}
\end{figure}

\begin{proof}
We denote the vertices of the standard tripod $T_0$ by $z_0,z_1,z_2$. By assumption there is a weakly $H$-quasisymmetric homeomorphism $\phi \colon T_0 \to T \subset p$ 
with $\diam(T) \geq H^{-1} \diam(p)$. By V\"ais\"al\"a's theorem \cite[Corollary 10.22]{Hei:01}, $\phi$ is in fact (strongly) quasisymmetric: there is a homeomorphism 
$\eta\colon [0,\infty) \to [0,\infty)$ depending only on $H$ so that for every $t>0$ and every triple of points $w_0,w_1,w_2$, 
$$
|w_1-w_0| \leq t|w_2-w_0| \quad \text{implies} \quad |\phi(w_1)-\phi(w_0)| \leq \eta(t)|\phi(w_2)-\phi(w_0)|. 
$$
We set $b_p:=\phi(0) \in p$. A standard quasisymmetric distortion estimate shows that if $\tau>0$ is small enough depending on $H$ then each of the three components of $T \setminus \{b_p\}$ intersect 
$\mathbb{S}(b_p,r_p)$. Thus, we can define the points 
$$
k_n= \min \{0 < s < 1: \, \phi(sz_n) \in \mathbb{S}(b_p,r_p)  \}, \quad n \in \{0,1,2\}, 
$$
and the curves 
$$
J_n=\phi([0,k_nz_n]), \quad n \in \{0,1,2\}. 
$$
Then $\mathbb{D}(b_p,r_p)\setminus \bigcup_{n=0}^2 J_n$ is the union of pairwise disjoint connected sets $V_0,V_1,V_2$ which are labeled so that  
$\overline{V_n} \cap J_n = \{b_p\}$. That is, $V_n$ is bounded by the other two curves $J_{n'}\neq J_n$ and a subarc $S_n$ of $\mathbb{S}(b_p,r_p)$. 

The arcs $S_n$ are pairwise disjoint. Thus, by changing the labeling if necessary, we may assume that 
\begin{equation} \label{ankleis}
\diam(S_0) \leq \sqrt{3}r_p < 2r_p.  
\end{equation} 

Next, another standard quasisymmetric distortion estimate shows that there are a point $a_p \in V_0$, satisfying $|a_p-b_p|=\tau^{3/4} r_p$, and a constant $0<C<1$ depending only on $H$, so that 
\begin{equation} \label{ohny} 
\mathbb{D}(a_p,C\tau^{3/4} r_p) \subset V_0. 
\end{equation} 
We require that $\tau^{3/4}+4 \tau \leq \tau^{1/2}$ and $4\tau \leq C \tau^{3/4}$. Then \eqref{ohny} shows that  
\begin{equation} \label{eiviel}
\mathbb{D}(a_p,4\tau r_p)\subset \mathbb{D}(a_p,C\tau^{3/4} r_p) \subset V_0.  
\end{equation}
Moreover, \eqref{indika} holds by the triangle inequality. 

Proposition \ref{blockingprop} now follows once we establish \eqref{dessed}. Fix $\alpha \in \Gamma(a_p,b_p,\tau)$ as in the proposition. In particular, $p \notin |\alpha|$. Since $\alpha(t) \subset \mathbb{D}(b_p,r_p)$ for every $s_1<t<s_2$, we conclude using \eqref{eiviel} and the assumption $|\alpha| \cap \mathbb{D}(a_p,4\tau r_p) \neq \emptyset$ that 
$$
\alpha(t) \subset V_0 \quad \text{for every } s_1 < t < s_2. $$ 
It follows that $\alpha(s_1)$ intersects the arc $S_0$. An analogous argument on $[t_2,t_1]$ shows that $\alpha(t_1)$ intersects $S_0$ as well. Thus, by \eqref{ankleis} we have 
\begin{equation}
\label{ankleis2}
\dist(\alpha(s_1),\alpha(t_1)) \leq \diam(S_0) \leq \sqrt{3}r_p. 
\end{equation} 
On the other hand, since $\diam(\alpha(s_1)) \leq \tau r_p$ 
and $\alpha(s_2) \subset \mathbb{D}(b_p,\tau^{1/2}r_p)$, 
by our assumption and \eqref{indika}, the triangle inequality yields  
$$ 
\dist(\alpha(s_1),\alpha(s_2)) \geq r_p-(\tau r_p+ \tau^{1/2}r_p) \geq \frac{99}{100} r_p
$$ 
if $\tau$ is required to be small enough so that the last inequality holds. The same argument shows that $s_1$ and $s_2$ may be replaced with $t_1$ and $t_2$. Combining the two estimates with \eqref{ankleis2} shows that \eqref{dessed} holds. The proof is complete.  
\end{proof}

\subsection{Good, bad, and large components} \label{gblsection}
We now describe the procedure of grouping the complementary components. We will  be able to construct an admissible function with bounded energy \eqref{queen} that assigns $\rho(p)=0$  for the ``bad'' components. 

We fix a point $a \in \mathbb{C}$ and a radius $R>0$, and recall that $\mathbb{A}(a,R)$ denotes the annulus $\mathbb{D}(a,4R)\setminus \overline{\mathbb{D}}(a,R/2)$. Our goal is to find an upper bound for the transboundary modulus of 
\[
\Gamma=\left\{\text{paths in } \pi_{\Omega}(\overline{\mathbb{A}}(a,R)) \text{ joining } \pi_{\Omega}(\mathbb{S}(a,4R)) \text{ and } \pi_{\Omega}(\mathbb{S}(a,R/2))\right\}.  
\] 
Both the transboundary modulus and the number $\tau$ in  Proposition~\ref{blockingprop} are invariant under translations and dilations, so we may assume $a=0$ and $R=1$. 

Let $P \subset \mathcal{C}(\Omega)$ be the collection of complementary components $p$ intersecting $\mathbb{A}:=\mathbb{D}(0,4)\setminus \overline{\mathbb{D}}(0,1/2)$, and let $0<\tau<\frac{1}{1000}$ be the constant in Proposition \ref{blockingprop}. We denote   
\begin{equation} \label{bai}
P_L=\{ p \in P : \, \diam(p) \geq \tau\} \quad \text{and} \quad P_S=P \setminus  P_L.  
\end{equation}

Recall that $H$ and $N$ are the constants in the Quasitripod and Packing Conditions (i) and (ii) of Theorem \ref{tripodkoebe}, respectively. 

\begin{lemma}\label{lem:mod-large}
We have $\mod(\Gamma_L) \le  100N \tau^{-2}$, where 
    \begin{equation*}
\Gamma_L:=\{\gamma \in \Gamma: \, \gamma \text{ passes through some $p \in P_L$} \}. 
\end{equation*}
In particular, the upper bound depends only on $H$ and $N$.
\end{lemma}
\begin{proof}
Observe that $\rho \colon \hat{\Omega}\to [0,\infty]$, $\rho=\chi_{P_L}$ is admissible for $\Gamma_L$ (notice that $\Gamma$ may include 
constant paths which happens if $p$ intersects both $\mathbb{S}(0,4)$ and $\mathbb{S}(0,1/2)$).  We cover $\overline{\mathbb{D}}(0,4)$ with 
$100 \tau^{-2}$ disks of radius $\tau$ and apply Packing Condition (ii) in Theorem \ref{tripodkoebe} (with constant $N$) to see that the cardinality of $P_L$ is bounded from above by $100N \tau^{-2}$. Therefore,  
$$
\mod(\Gamma_L) \le \sum_{p \in \mathcal{C}(\Omega)} \rho(p)^2 \leq 100N \tau^{-2}.      
$$
Since $\tau$ (from Proposition~\ref{blockingprop}) depends only on $H$, the proof is complete. 
\end{proof}
Applying the subadditivity of the transboundary modulus (Proposition \ref{transproperties} (2)), we conclude from Lemma \ref{lem:mod-large} that in order to prove Theorem \ref{thm:mainestimate} it remains to consider 
\begin{equation} \label{keisti}
\Gamma_S:=\{\gamma \in \Gamma: \, \gamma \text{ does not pass through any } p \in P_L \}. 
\end{equation} 

We apply Proposition \ref{blockingprop} to find a suitable partition of $P_S$ into ``good'' and ``bad'' components. Given $p \in P_S$, let $r_p=\tau \diam(p)$ and $a_p \in \mathbb{C}$, 
$b_p \in p$, be as in Proposition \ref{blockingprop}. We start by choosing $p_1 \in P_S$ so that 
\begin{equation} \label{baiii}
\diam(p_1)=\max_{p \in P_S} \diam(p). 
\end{equation} 
Denote $r_1:=r_{p_1}$, $a_1:=a_{p_1}$ and $b_1:=b_{p_1}$, and let 
\begin{eqnarray*}
G_1 &:=& \left\{p \in P_S   : \diam(p) \geq \tau r_1,\, \dist(p,p_1) \leq \tau^{-2}r_1=\tau^{-1} \diam(p_1) \right\}, \text{ and } \\ 
B_1 &:=& \left\{p \in P_S   : \diam(p) < \tau r_1, \, a_p \in \overline{\mathbb{D}}(a_1,2\tau r_1) \right\}.   
\end{eqnarray*}
Then $G_1$ consists of the ``small'' boundary components $p \in P_S$ that are not too small relative to $p_1$ and not too far from $p_1$, while $B_1$ consists of the components $p \in P_S$ that are very small relative to $p_1$ and are located close to the ``center'' of the (fixed) quasitripod contained in $p_1$.

Suppose then that $p_\ell \in P_S$ and $G_\ell, B_\ell \subset P_S$ are chosen for $1 \leq \ell \leq k$. 
We stop the process if $P_S \setminus \bigcup_{\ell=1}^k (G_\ell \cup B_\ell) =\emptyset$. Otherwise, we choose $p_{k+1} \in P_S \setminus \bigcup_{\ell=1}^k (G_\ell \cup B_\ell)$ so that 
 \[ 
\diam(p_{k+1})=\max_{p \in P_S \setminus \bigcup_{\ell=1}^k (G_\ell \cup B_\ell)} \diam(p). 
\] 
We denote $r_{k+1}:=r_{p_{k+1}}$, $a_{k+1}:=a_{p_{k+1}}$ and $b_{k+1}:=b_{p_{k+1}}$, and let 
\begin{eqnarray*}
G_{k+1} := \left\{p \in P_S \setminus \bigcup_{\ell=1}^{k} (G_\ell \cup B_\ell)  : \diam(p) \geq \tau r_{k+1},\, \dist(p,p_{k+1}) \leq \tau^{-2} r_{k+1} \right\},  \\ 
\end{eqnarray*} 
\text{ and }
\begin{eqnarray*}
B_{k+1} := \left\{p \in P_S \setminus \bigcup_{\ell=1}^{k} (G_\ell \cup B_\ell)  : \diam(p) < \tau r_{k+1}, \, a_p \in \overline{\mathbb{D}}(a_{k+1},2\tau r_{k+1}) \right\}.  
\end{eqnarray*} 

Since $p_{k+1} \in G_{k+1}$ and $\Omega$ is finitely connected, the process stops after $L<\infty$ steps and we have a partition of $P_S$ into disjoint sets $G_k$ and  $B_k$, $k=1,\ldots, L$, so,
\begin{equation} \label{vasy}
P_S= G \cup B, \quad \text{where} \quad G:= \bigcup_{k=1}^L G_k, \quad B:= \bigcup_{k=1}^L B_k.
\end{equation}
It is worth emphasizing the fact that, by construction, each $G_k$ comes with a distinguished element $p_k$ and its associated $a_k, b_k$ and $r_k$. 

We now make precise the notion that bad components are ``overshadowed'' by the good ones. If $p \in B_{k}$, for some $k=1,\ldots, L$, then 
$$ 
a_p \in \overline{\mathbb{D}}(a_{k},2\tau r_{k}) \cap \mathbb{D}(b_p,\tau^{1/2}r_p) \subset 
\overline{\mathbb{D}}(a_{k},2\tau r_{k}) \cap \mathbb{D}(b_p,\tau^{5/2}r_{k})  
$$
by \eqref{indika} and since $r_p < \tau^2 r_k$. Since $b_p \in p$, it follows that 
$$
\dist(a_{k},p) \leq |a_{k}-a_p|+\dist(a_p,p) 
\leq 2\tau r_{k} + \tau^{5/2} r_{k} < 3\tau r_{k}. 
$$ 
Since $\diam(p)<\tau r_{k}$, we conclude that 
\begin{equation} \label{karkaamo}
p \subset \mathbb{D}(a_{k},4 \tau r_{k}) \quad \text{for every } p \in B_{k}. 
\end{equation}

We will construct a suitable admissible function $\rho$ for $\Gamma_S$ which equals zero in $B$. As mentioned already, this will be compensated by the weights on the 
elements $p \in G$ and/or the costs of avoiding such elements. The following lemma provides a bound on the number of good components at a given scale, and hence will be useful in later modulus estimates. 

\begin{lemma}\label{inter} The disks 
$\mathbb{D}(a_k,\tau r_k)$, $1 \leq k \leq L$, 
are pairwise disjoint. Moreover, if $1\leq m < k \leq L$ and 
\begin{equation}\label{onon}
\mathbb{D}(a_m, 2r_m) \cap \mathbb{D}(a_k, 2r_k) \neq \emptyset,  
\end{equation} 
then $r_k < \diam(p_k) < \tau r_m$. 
\end{lemma} 
\begin{proof}
We start with the second claim. The triangle inequality shows that  
\begin{equation} \label{kina}
\dist(p_m,p_k) \leq \dist(a_k,p_k)+\dist(p_m,a_m)+|a_m-a_k|. 
\end{equation} 
By \eqref{indika} and \eqref{onon}, we have 
\begin{equation} \label{mittaa}
\dist(a_k,p_k)\leq r_k, \quad \dist(a_m,p_m) \leq r_m \quad \text{and} \quad |a_m-a_k| \leq 2(r_k+r_m). 
\end{equation}
From $m<k$ it follows that $r_k \leq r_m$. Therefore, combining \eqref{kina} and \eqref{mittaa} we have 
\begin{equation}\label{kiina} 
\dist(p_m,p_k) \leq 3(r_k+r_m) \leq 6 r_m. 
\end{equation}
Since $p_k \notin \bigcup_{\ell=1}^m(G_\ell \cup B_\ell)$, by the definition of $G_m$ and \eqref{kiina} we have $r_k < \diam(p_k) <\tau r_m$. 

To prove the first claim, we assume towards a contradiction that there are numbers $1 \leq m < k \leq L$ so that \begin{equation}\label{onon2}
\mathbb{D}(a_m, \tau r_m) \cap \mathbb{D}(a_k, \tau r_k) \neq \emptyset.   
\end{equation} 
Then \eqref{onon} holds, and so $\diam(p_k)<\tau r_m$ by the second claim. Thus, since $p_k \notin B_m$, we have $a_k \notin \overline{\mathbb{D}}(a_m,2\tau r_m)$ by the definition of $B_m$. We conclude using the triangle inequality that if $z \in \mathbb{D}(a_k,\tau r_k)$, then 
$$
|z-a_m| \geq |a_k-a_m|-\tau r_k \geq 2\tau r_m-\tau r_k >\tau r_m, 
$$
which contradicts \eqref{onon2}. The proof is complete. 
\end{proof} 
\subsection{Modulus bound and the proof of Theorem~\ref{thm:mainestimate}}
Our goal is to give an upper bound for $\Mod \Gamma_S$, where $\Gamma_S$ is defined in \eqref{keisti}. Recall the ``good'' and ``bad'' sets $G$ and $B$ in \eqref{vasy}. 
Note that if a non-negative Borel function $\rho$ is admissible for the family of \emph{injective} paths in $\Gamma_S$, then $\rho$ is admissible for $\Gamma_S$. Indeed, for every rectifiable $\gamma_2 \in \Gamma_S$ there is an injective $\gamma_1 \in \Gamma_S$ so that $|\gamma_1| \subset |\gamma_2|$, see e.g. \cite[Proposition 15.1]{Sem:96:PI}. Then, if $\rho$ is admissible for injective paths, we have  
$$
\int_{\gamma_2} \rho \, ds \geq \int_{\gamma_1} \rho \, ds \geq 1,  
$$
so $\rho$ is admissible for $\Gamma_S$. 

We fix an injective $\gamma_1 \in \Gamma_S$. After reparameterization and recalling that $\gamma_1$ does not pass through any $p \in P$ with diameter greater than $\tau < \frac{1}{1000}$, we may assume, without loss of generality, that the domain of $\gamma_1$ contains $[0,1]$, $\gamma_1([0,1]) 
\subset \pi_{\Omega}(\mathbb{D}(0,3))$, and  
$$
\gamma_1(0) \in \Omega \cap \mathbb{D}(0,3) \setminus \mathbb{D}(0,5/2), \quad  \quad \gamma_1(1) \in \Omega \cap \mathbb{D}(0,3/4). 
$$
Given a path $\alpha \colon I \to \hat{\Omega}$, we denote 
\begin{equation} \label{elka}
G(\alpha):=\{ p:  p \in G \cap |\alpha| \}. 
\end{equation} 

\begin{proposition}\label{goodpathsprop} 
Let $\gamma:=\gamma_1|[0,1]$. Then there exist intervals $[c_\nu,d_\nu] \subset [0,1]$, $\nu=1,2,\ldots, \mu$, with non-empty and pairwise disjoint interiors so that 
$\gamma(t) \notin B$ for every $t \in \bigcup_{\nu=1}^\mu (c_\nu,d_\nu)$ and  
\begin{equation} \label{kiikos}
1 \leq \sum_{\nu=1}^\mu \dist(\gamma(c_\nu),\gamma(d_\nu)) +\frac{11}{\tau} \sum_{p \in G(\gamma)} \diam(p). 
\end{equation} 
\end{proposition} 
Here we have applied the notation $G(\alpha)$ introduced in \eqref{elka}. We postpone the proof of Proposition \ref{goodpathsprop} and first show how it completes the proof of Theorem~\ref{thm:mainestimate}. As discussed after Lemma \ref{lem:mod-large}, it remains to prove an upper bound for $\Mod(\Gamma_S)$ that depends only on $H$ and $N$. 

Let $M=\frac{11}{\tau}$ be the number which appears in \eqref{kiikos}. Define $\rho:\hat{\Omega} \to [0,\infty]$,
\begin{eqnarray}
\label{testfunction} 
\rho(p)=\left\{  \begin{array}{ll}
1, & p \in \Omega \cap \mathbb{D}(0,3), \\
(M+1)\diam(p), & p \in G, \\ 
0, & \text{otherwise}.   
\end{array} 
\right. 
\end{eqnarray} 

\begin{remark} The admissibility of $\rho$ for $\Gamma_S$, which is proved below, is the key to Theorem \ref{tripodkoebe} (see e.g. \cite{HakLi23,MTW13} for similar constructions). The difficulty is that a given path $\gamma \in \Gamma_S$ may pass through components $p \in B$ where $\rho(p)=0$ and thus use them as ``shortcuts''. However, Proposition \ref{goodpathsprop} allows us to sacrifice the ``bad parts'' of $\gamma$ which may intersect $B$, and conclude the admissibility of $\rho$ by only considering the ``good parts''. 
\end{remark}

We now apply Proposition \ref{goodpathsprop} to prove the admissibility of $\rho$ for $\Gamma_S$. Let the path $\gamma$ and the intervals $[c_\nu,d_\nu]$ be as in the proposition. Since $\gamma(t) \notin B$ 
for every $c_\nu < t < d_\nu$ and since $|\gamma| \subset \mathbb{D}(0,3)$, the triangle inequality gives  
\begin{equation}\label{yksnuu}
\dist(\gamma(c_\nu),\gamma(d_\nu)) \leq \int_{\gamma|[c_\nu,d_\nu]} \rho \, ds + \sum_{p \in G(\gamma|(c_\nu,d_\nu))} \diam(p) 
\end{equation} 
for every $1 \leq \nu \leq \mu$. Recall that the integral in \eqref{yksnuu} is over the subpaths of $\gamma|[c_\nu,d_\nu]$ whose images are in $\Omega$. 
Since $\gamma$ is injective and the intervals $[c_\nu,d_\nu]$ have disjoint interiors, summing \eqref{yksnuu} over $\nu$ gives 
\begin{equation}
\label{kakspuu}
\sum_{\nu=1}^\mu \dist(\gamma(c_\nu),\gamma(d_\nu)) \leq \int_{\gamma} \rho \, ds +  \sum_{p \in G(\gamma)} \diam(p). 
\end{equation}
Combining \eqref{kakspuu} and Proposition \ref{goodpathsprop} shows that $\rho$ is admissible for $\Gamma_S$. 

We estimate the energy $\int_{\Omega} \rho^2 \, dA + \sum_{p \in \mathcal{C}(\Omega)} \rho(p)^2$. Given $1 \leq k \leq L$, recall that every $p \in G_k$ satisfies 
\begin{equation}\label{boo1}
\tau r_k \leq \diam(p) \leq \diam(p_k) = \tau^{-1} r_k.  
\end{equation} 
We claim that 
\begin{equation} \label{edista} p \subset \mathbb{D}(a_k,\tau^{-3}r_k) \quad \text{for all }p \in G_k.
\end{equation} 
Indeed, by \eqref{boo1} it suffices to show that 
$\dist(a_k,p) \leq 2 \tau^{-2}r_k$. 

We have $b_k \in p_k$ and $|a_k-b_k| \leq \tau^{1/2}r_k$ by \eqref{indika}, and $\dist(p,p_k) \leq \tau^{-2} r_k$ by the definition of $G_k$. Thus, applying \eqref{boo1} and the triangle inequality, we conclude that 
\begin{eqnarray*}
\dist(a_k,p) &\leq& |a_k-b_k| +\diam(p_k) + \dist(p_k,p) \\
&\leq& \tau^{1/2} r_k +\tau^{-1}r_k+\tau^{-2}r_k \leq 
2\tau^{-2}r_k. 
\end{eqnarray*}
We have proved \eqref{edista}. 

We cover $\mathbb{D}(a_k,\tau^{-3}r_k)$ with disks $D_{m,n}$ of radius $\tau r_k$ and centers 
$$
a_k+m\tau r_k+i(n\tau r_k), \quad -\tau^{-4} \leq m,n \leq \tau ^{-4}. 
$$ 
Notice that there are fewer than $10\tau^{-8}$ 
of such disks. 

Applying the first inequality in \eqref{boo1}, \eqref{edista}, and Packing Condition (ii) in Theorem \ref{tripodkoebe} (with constant $N$) to each of the disks $D_{m,n}$, shows that 
\begin{equation}\label{boo2}
\card G_k \leq 10N\tau^{-8} \quad \text{for every } 1 \leq k \leq L.  
\end{equation}

We now estimate the energy of $\rho$. Since $\int_{\Omega} \rho^2 \, dA \leq |\mathbb{D}(0,3)| = 9 \pi$, it suffices to estimate the sum of $\rho^2$ over $\mathcal{C}(\Omega)$. By \eqref{boo1} and \eqref{boo2} we have 
\begin{eqnarray} \label{trap}
\sum_{p \in G_k} \rho(p)^2 
&=& (M+1)^2 \sum_{p \in G_k} \diam(p)^2 \\ 
\nonumber &\leq&  2M^2(\card G_k) \diam(p_k)^2  
\leq 20M^2N\tau^{-10}r_k^2  
\end{eqnarray} 
for every $1 \leq k \leq L$ (notice that since $M \geq 10$, we have $(M+1)^2 \leq 2M^2$). On the other hand, the disks 
$\mathbb{D}(a_k,\tau r_k)$, $1 \leq k \leq L$, are pairwise disjoint subsets of $\mathbb{D}(0,5)$ by Lemma \ref{inter}. Thus,   
\begin{equation} \label{trap2}
\pi \tau^2\sum_{k=1}^L r_k^2=\sum_{k=1}^L |\mathbb{D}(a_k,\tau r_k)| \leq |\mathbb{D}(0,5)|= 25\pi.  
\end{equation} 
Combining \eqref{trap} and \eqref{trap2} yields
$$
\sum_{p \in G} \rho(p)^2 \leq 500 M^2N \tau^{-12}. 
$$ 
In conclusion, $\int_{\Omega} \rho^2 \, dA + \sum_{p \in \mathcal{C}(\Omega)} \rho(p)^2$ is bounded from above by a constant that depends only on $N$ and $H$. 

We have proved that Theorem \ref{thm:mainestimate} follows from Proposition \ref{goodpathsprop}. 

\subsection{Proof of Proposition \ref{goodpathsprop}: Finding good subpaths} \label{goodpathssec}

Recall that $\gamma \colon [0,1] \to \hat{\Omega}$ is an injective path which does not pass through any large component $p \in P_L$. Moreover, $\gamma([0,1]) 
\subset \pi_{\Omega}(\mathbb{D}(0,3))$, and  
$$
\gamma(0) \in \Omega \cap \mathbb{D}(0,3) \setminus \mathbb{D}(0,5/2), \quad  \quad \gamma(1) \in \Omega \cap \mathbb{D}(0,3/4). 
$$
To prove Proposition \ref{goodpathsprop}, we need to find the segments $[c_\nu,d_\nu]$ as in the proposition such that $\gamma(t) \notin B$ for every $c_\nu < t < d_\nu$ and with $M=\frac{11}{\tau}$ 
\begin{equation}\label{oots}
1 \leq \sum_{\nu=1}^\mu \dist(\gamma(c_\nu),\gamma(d_\nu)) +M \sum_{p \in G(\gamma)} \diam(p). 
\end{equation}
We may assume that $\gamma(t) \in B$ for some $0<t<1$, since otherwise Proposition \ref{goodpathsprop} follows by choosing $\mu=1$ and $[c_1,d_1]=[0,1]$. 

We construct the collection of segments 
$$
\mathcal{I}_L=\{[c_\nu,d_\nu], \, 1 \leq \nu \leq \mu\}, 
$$
inductively by starting with $\mathcal{I}_0=\{[0,1]\}$. At step $1 \leq k \leq L$, we choose a collection of subsegments $I_k$ of the segments $I_{k-1} \in \mathcal{I}_{k-1}$ by suitably removing any overlap of $\gamma(I_{k-1})$ and $B_k$, so that we can eventually apply Proposition \ref{blockingprop} to compensate for the ``loss'' of the elements $p \in B_k$ and establish \eqref{oots}. 

For the construction, it is useful to recall that if $p \in B_k$ then 
$$
p \subset \mathbb{D}(b_k,\tau^{1/2}r_k) \subset 
\mathbb{D}(b_k,r_k/30) \subset \mathbb{D}(b_k,r_k). 
$$
Suppose that the collections $\mathcal{I}_\ell$ are defined for $0 \leq \ell \leq k-1$. Fix $[s_0,t_0] \in \mathcal{I}_{k-1}$ and denote $\alpha=\gamma|[s_0,t_0]$. We consider the following cases: 

\begin{enumerate}
\item If $\alpha(t) \notin B_k$ for all $s_0 <t < t_0$, then we include $[s_0,t_0]$ in $\mathcal{I}_{k}$. 
\item Otherwise, let ($B_k$ is a finite set and so $s_0< s_2\leq t_2 < t_0$ below) 
\begin{eqnarray*}
A &=& \{s_0 < t < t_0: \, \alpha(t) \in B_k\}, \quad s_2=\min A \quad \text{and} \quad t_2=\max A, \\
A_2 &=& \{ s_0 <t<s_2: \, \alpha(t) \cap \mathbb{S}(b_k,r_k) \neq \emptyset\}, \quad \text{and }\\ 
A_3 &=& \{ t_2 <t<t_0: \, \alpha(t) \cap \mathbb{S}(b_k,r_k) \neq \emptyset\}. 
\end{eqnarray*}
\end{enumerate}
Then $A_2 =\emptyset$ if $\alpha$ ``meets'' $B_k$ before $\mathbb{S}(b_k,r_k)$, and $A_3=\emptyset$ if $\alpha$ ``exits'' $B_k$ after $\mathbb{S}(b_k,r_k)$. 

\begin{itemize}
\item[(a1)] If $A_2 \cup A_3=\emptyset$, we do not include any subinterval of $[s_0,t_0]$ in $\mathcal{I}_k$. In this case we have 
\begin{equation} \label{berb-a1}
\alpha(s_0) \cup \alpha(t_0) \subset \mathbb{D}(b_k,r_k). 
\end{equation} 
\item[(a2)] If $A_2\neq \emptyset$ and $A_3=\emptyset$, we include $[s_0,s_2]$ in $\mathcal{I}_k$. In this case we have 
\begin{equation} \label{berb-a2}
\alpha(t_0) \subset \mathbb{D}(b_k,r_k). 
\end{equation}
\item[(a3)] If $A_2 = \emptyset$ and $A_3 \neq \emptyset$, we include $[t_2,t_0]$ in $\mathcal{I}_k$. In this case we have 
\begin{equation} \label{berb-a3}
\alpha(s_0) \subset \mathbb{D}(b_k,r_k). 
\end{equation} 
\item[(b)] If $A_2 \neq \emptyset$ and $A_3\neq \emptyset$, let $s_1 = \max A_2$ and $t_1=\min A_3$. Notice that $s_0 <s_1<s_2 \leq t_2 <t_1 <t_0$. 
\begin{itemize} 
\item[(b1)] if $\max\{\diam(\alpha(s_1)), \diam(\alpha(t_1))\} \ge \tau r_k$, we include $[s_0,s_1]$ and $[t_1,t_0]$ in $\mathcal{I}_k$.  
\item[(b2)] Otherwise we include $[s_0,s_1]$, $[s_1,s_2]$, $[t_2,t_1]$ and $[t_1,t_0]$ in $\mathcal{I}_k$. 
\end{itemize}
\end{itemize}
\vskip10pt 

Let $\mathcal{I}_k([s_0,t_0])$ be the family of subsegments of $[s_0,t_0] \in \mathcal{I}_{k-1}$ included in $\mathcal{I}_k$ using the above algorithm, and set
$$
\mathcal{I}_k=\bigcup_{[s_0,t_0] \in \mathcal{I}_{k-1}} \mathcal{I}_k([s_0,t_0]), \quad 1 \leq k \leq L. 
$$ 

The above construction and a simple induction argument show that 
\begin{equation} \label{markku}
\gamma((c,d)) \cap \Big(\bigcup_{\ell=1}^kB_\ell \Big)= \emptyset  \quad \text{for all } [c,d] \in \mathcal{I}_k, \, 1 \leq k \leq L. 
\end{equation} 

In particular, $\gamma(t) \notin B$ for all 
$t \in \bigcup_{[c,d] \in \mathcal{I}_L} (c,d)$. The interiors of distinct segments in $\mathcal{I}_L$ are non-empty and pairwise disjoint. Thus, in order to prove Proposition \ref{goodpathsprop} it suffices to show that the segments in $\mathcal{I}_L$ satisfy \eqref{oots}. 

Given $1 \leq k \leq L$, let $\mathcal{J}_{k-1}(e) \subset \mathcal{I}_{k-1}$ {denote} the family {of intervals in $\mathcal{I}_{k-1}$} for which case 
$$
e \in \{(1),(a1),(a2),(a3),(b1),(b2)\}$$ 
applies {in the algorithm above. Similarly, set}
\begin{eqnarray*} 
\mathcal{J}_{k-1}(a) &=& \mathcal{J}_{k-1}(a1) \cup \mathcal{J}_{k-1}(a2) \cup \mathcal{J}_{k-1}(a3), \\ \mathcal{J}_{k-1}(b) &=& \mathcal{J}_{k-1}(b1) \cup \mathcal{J}_{k-1}(b2), \quad \text{and} \quad  
\mathcal{J}(e)=\bigcup_{k=1}^L \mathcal{J}_{k-1}(e). 
\end{eqnarray*}

We next claim that  
\begin{eqnarray} \nonumber 
\frac{11}{10} \leq \sum_{I \in \mathcal{I}_{L}} T(I) 
&+& \sum_{k=1}^L \big(2 (\card\mathcal{J}_{k-1}(a))-\frac{1}{9} (\card\mathcal{J}_{k-1}(b2)) \big) \cdot r_k \\
\label{forjes}
&+& \frac{3}{\tau} \sum_{p \in G(\gamma)} \diam(p), 
\end{eqnarray}
where we use the notation  
$$
T(I)=\dist(\gamma(c),\gamma(d)), \quad I=[c,d]. 
$$ 

\subsection{Proof of Proposition \ref{goodpathsprop}: Preliminary estimates} 

The goal of this subsection is to establish \eqref{forjes}. 

\begin{lemma} \label{alemma} 
Let $1 \leq k \leq L$ and $[s_0,t_0] \in \mathcal{J}_{k-1}(a)$. Then 
$$
\dist(\gamma(s_0),\gamma(t_0)) \leq Q([s_0,t_0])+2 r_k, 
$$ 
where 
\begin{eqnarray*}
Q([s_0,t_0])=\left\{  \begin{array}{ll}
0 & \text{in Case (a1)}, \\
\dist(\gamma(s_0),\gamma(s_2)) & \text{in Case (a2)}, \\
\dist(\gamma(t_2),\gamma(t_0)) & \text{in Case (a3)}. 
\end{array}
\right. 
\end{eqnarray*} 
\end{lemma}

\begin{proof}
Case (a1) follows from \eqref{berb-a1} and the triangle inequality. In Case (a2), since $\gamma(s_2) \in B_k$, we have 
$$
\gamma(s_2) \subset \mathbb{D}(a_k,4\tau r_k) \subset \mathbb{D}(b_k,\tau^{1/2}r_k)
$$ 
by \eqref{indika} and \eqref{karkaamo}, and $\diam(\gamma(s_2)) < \tau r_k$. Therefore, applying \eqref{berb-a2} yields  
\begin{eqnarray*}
\dist(\gamma(s_0),\gamma(t_0)) &\leq& \dist(\gamma(s_0),\gamma(s_2))+\dist(\gamma(s_2),\gamma(t_0)) + \diam(\gamma(s_2)) \\
&\leq& \dist(\gamma(s_0),\gamma(s_2)) + (1+\tau^{1/2})r_k + \tau r_k \\
&\leq&  \dist(\gamma(s_0),\gamma(s_2))+2 r_k 
\end{eqnarray*} 
by the triangle inequality and since $\tau^{1/2}+\tau \leq 1$. 
Case (a3) follows in the same way by applying \eqref{berb-a3}. 
\end{proof}

\begin{lemma} \label{b1lemma}
Let $1 \leq k \leq L$ and $[s_0,t_0] \in \mathcal{J}_{k-1}(b1)$. Then 
\begin{equation} \label{doomi}
\diam(\gamma(c)) \geq \tau r_k \quad \text{and} \quad \gamma(c) \in \bigcup_{\ell=1}^k G_\ell, 
\quad \text{for } c=s_1 \text{ or } c=t_1. 
\end{equation} 
Moreover,  
\begin{equation} \label{ecan}
\dist(\gamma(s_0),\gamma(t_0)) \leq \dist(\gamma(s_0),\gamma(s_1))+\dist(\gamma(t_1),\gamma(t_0))+ \frac{3}{\tau} D([s_0,t_0]),   
\end{equation} 
where 
$$
D([s_0,t_0])=\sum \diam(p), 
$$ 
and the sum is over those $p\in\{\gamma(s_1),\gamma(t_1)\}$ which satisfy \eqref{doomi}.  
\end{lemma}

\begin{proof}
Recall that both $\gamma(s_1),\gamma(t_1)$ intersect $\mathbb{S}(b_k,r_k)$ and 
\begin{equation} 
\label{toimi}
\diam(\gamma(c)) \geq \tau r_k \quad \text{for } c=s_1 \text{ or } t_1. 
\end{equation}
Also, recall from \eqref{markku} that $s_0<s_1<t_1<t_0$ and 
$$
\gamma(t) \notin \bigcup_{\ell=1}^{k-1}B_\ell \quad \text{for all } s_0<t<t_0.  
$$ 
Thus, the definition of $G_k$ shows that if $c$ satisfies \eqref{toimi} then $\gamma(c) \in \bigcup_{\ell=1}^k G_\ell$. By the triangle inequality we have 
\begin{eqnarray*}
\dist(\gamma(s_0),\gamma(t_0)) &\leq& \dist(\gamma(s_0),\gamma(s_1))+\dist(\gamma(t_1),\gamma(t_0))\\ 
&+&\dist(\gamma(s_1),\gamma(t_1)) 
+ \diam(\gamma(s_1))+\diam(\gamma(t_1)). 
\end{eqnarray*} 
The last distance is bounded from above by $2r_k \leq 2 \tau^{-1}D([s_0,t_0])$, and the sum of the diameters is bounded from above by 
$\tau r_k + D([s_0,t_0])$ which is at most $ 2D([s_0,t_0]) \leq \tau^{-1}D([s_0,t_0])$. The inequality \eqref{ecan} follows.  
\end{proof}

\begin{lemma} \label{b2lemma}
Let $1 \leq k \leq L$ and $[s_0,t_0] \in \mathcal{J}_{k-1}(b2)$. Then 
$$ 
\dist(\gamma(s_0),\gamma(t_0)) \leq \sum_{m=0}^1 \Big[\dist(\gamma(s_m),\gamma(s_{m+1}))+\dist(\gamma(t_m),\gamma(t_{m+1})) \Big] -\frac{1}{9}r_k. 
$$
\end{lemma}
\begin{proof}
Recall that $p \subset \mathbb{D}(a_k,4\tau r_k)$ for every $p \in B_k$ by \eqref{karkaamo}. Therefore, the path $\alpha=\gamma{|[s_1,t_1]}$ satisfies conditions (i)-(iv) preceding Proposition \ref{blockingprop}, allowing us to apply the proposition to show that 
$$
\dist(\gamma(s_1),\gamma(t_1)) \leq \dist(\gamma(s_1),\gamma(s_{2}))+\dist(\gamma(t_1),\gamma(t_2))  -\frac{1}{10} r_k. 
$$
We also have $\diam(\gamma(s_1))+\diam(\gamma(t_1))\leq 2\tau r_k<r_k$ by assumption. The claim follows by combining the estimates with the triangle inequality. 
\end{proof}

We are ready to prove \eqref{forjes}. We apply Lemmas \ref{alemma}, 
\ref{b1lemma} and \ref{b2lemma} to see that if $1 \leq k \leq L$ then (recall the notation $T(I)=\dist(\gamma(a),\gamma(b))$ for $I=[a,b]$)  
\begin{eqnarray}
\nonumber \sum_{I' \in \mathcal{I}_{k-1}} T(I') &\leq& \sum_{I \in \mathcal{I}_{k}} T(I) + 
\big(2(\card\mathcal{J}_{k-1}(a))-\frac{1}{9} (\card\mathcal{J}_{k-1}(b2)) \big) \cdot r_k \\  
\label{biggo} &+& \frac{3}{\tau} \sum_{I \in \mathcal{J}_{k-1}(b1)} D(I).  
\end{eqnarray} 
Recalling that $T([0,1]) \geq \dist(\gamma(0),\gamma(1)) \ge 5/2-3/4 > \frac{11}{10}$  
and applying induction together with \eqref{biggo} yields 
\begin{eqnarray} \nonumber
\frac{11}{10} \leq \sum_{I \in \mathcal{I}_{L}} T(I) 
&+& \sum_{k=1}^L \big(2 (\card\mathcal{J}_{k-1}(a))-\frac{1}{9} (\card\mathcal{J}_{k-1}(b2)) \big) \cdot r_k \\ 
\label{muhku} 
&+& \frac{3}{\tau} \sum_{k=1}^L \sum_{I \in \mathcal{J}_{k-1}(b1)} D(I).   
\end{eqnarray}
Finally, it follows from the construction that each $p \in G$ satisfies \eqref{doomi} in Lemma \ref{b1lemma} for at most one interval $[s_0,t_0] \in \mathcal{J}(b1)$. Therefore    
\begin{equation} \label{boomtrack} 
\sum_{k=1}^L \sum_{I \in \mathcal{J}_{k-1}(b1)} D(I) =  \sum_{I \in \mathcal{J}(b1)} D(I) \leq \sum_{p \in G(\gamma)} \diam(p);  
\end{equation} 
recall that $G(\gamma)=\{p: p \in G \cap |\gamma|\}$.  Combining \eqref{muhku} and \eqref{boomtrack} proves \eqref{forjes}. 


\subsection{Proof of Proposition \ref{goodpathsprop}: Estimates for consecutive segments} \label{cssection}
Estimate \eqref{kiikos}, which is the remaining claim in Proposition \ref{goodpathsprop}, follows from combining \eqref{forjes} with   
\begin{eqnarray} \nonumber  
\sum_{k=1}^L (\card\mathcal{J}_{k-1}(a)) \cdot r_k &\leq& \frac{1}{20}  + \frac{4}{\tau}\sum_{p \in G(\gamma)} \diam(p)\\ \label{Feet} &+& 12 \tau \sum_{k=1}^L (\card\mathcal{J}_{k-1}(b2))\cdot r_k.   
\end{eqnarray}
The rest of Section \ref{sec:proofestimate} is devoted to the proof of \eqref{Feet}. The sum on the left represents the ``loss'' incurred by removing the ``bad'' subsegments from the path $\gamma$ in Section \ref{goodpathssec}; these are the subsegments which are removed without getting any ``direct compensation'' (Cases (a) in Section \ref{goodpathssec}). The right side of the inequality represents the (indirect) ``compensation'', which will be obtained by associating to every $I \in \mathcal{J}(a)$ a ``good'' $I' \in \mathcal{J}(b)$ (or $[0,1]$, which yields the constant $20^{-1}$) from an earlier ``generation''. To such intervals $I'$ we can apply the good components $p \in G$ (case (b1) in Section \ref{goodpathssec}), or the ``detour'' Proposition \ref{blockingprop} (case (b2)). 

We find the intervals $I'$ by analyzing certain sequences of consecutive subintervals in the construction of $\mathcal{I}_L$. We now give precise definitions. We say that $J \in \mathcal{I}_k$ is a \emph{child} of 
$I \in \mathcal{I}_{k-1}$, and $I$ the \emph{parent} of $J$, if $J \subset I$. The definitions of grandchildren and grandparents are then obvious. An interval $I$ may be its own child and parent. That is, 
$I$ may be an element of both $\mathcal{I}_k$ and $\mathcal{I}_{k'}$ for $k \neq k'$. 

We recall the cases in Section \ref{goodpathssec}. 
Let $1 \leq k \leq L$. Then any $[s_0,t_0]$ in 
\begin{itemize}
\item[(1)] $\mathcal{J}_{k-1}(1)$ has one child, namely itself; $[s_0,t_0] \in \mathcal{I}_{k}$.  
\item[(a1)] $\mathcal{J}_{k-1}(a1)$ does not have any children. 
\item[(a2)] $\mathcal{J}_{k-1}(a2)$ has one child $[s_0,s_2]\in \mathcal{I}_{k}$. 
\item[(a3)] $\mathcal{J}_{k-1}(a3)$ has one child $[t_2,t_0]\in \mathcal{I}_{k}$. 
\item[(b1)] $\mathcal{J}_{k-1}(b1)$ has two 
children $[s_0,s_1],[t_1,t_0] \in \mathcal{I}_{k}$. 
\item[(b2)] $\mathcal{J}_{k-1}(b2)$ has four children $[s_0,s_1],[s_1,s_2],[t_2,t_1],[t_1,t_0] \in \mathcal{I}_k$. 
\end{itemize} 
In summary, the segments in $\mathcal{J}(a1)$ do not have children, while all other segments in $\bigcup_{k=1}^{L} \mathcal{I}_{k-1}$ have at least one child. 

To start the proof of \eqref{Feet}, we notice that 
$[0,1] \notin \mathcal{J}(a1)$. Indeed, we assume (before Proposition \ref{goodpathsprop}) that $\dist(\gamma(0),\gamma(1)) \geq 1$. On the other hand, by the definition of the ``small'' components $P_S$ in \eqref{bai} and the radius $r_1$ in \eqref{baii} and after \eqref{baiii}, we have 
\begin{equation}  
\label{triviii} 
r_1 \leq \tau^{2}<\frac{1}{1000000}. 
\end{equation} 
Thus \eqref{berb-a1} cannot hold, so case (a1) cannot happen. 

Moreover, if $L=1$ then \eqref{Feet} holds since 
$[0,1]$ is the only element of $\mathcal{I}_0$ and the sum on the left of \eqref{Feet} is at most $r_1$, which is at most $20^{-1}$ by \eqref{triviii}. 

We assume from now on that $L \geq 2$. We next define a finite sequence $S(I)$ for every segment $I \in \mathcal{I}_{\ell}$ which is of one of the following two types: 
\vskip 10pt
\begin{itemize} 
\item[($\ell=1$)] $I$ is any child of $[0,1]$. 
\vskip 5pt
\item[($\ell \geq 2$)] The parent $J$ of $I$ is in $\mathcal{J}_{\ell-1}(b)$. Moreover, $J$ and $I$ have different left endpoints. 
\end{itemize} 
\vskip 10pt
Recall that a segment $I$ may be an element of two different collections $\mathcal{I}_k$ and $\mathcal{I}_{k'}$, $k \neq k'$. However, 
\vskip5pt
\begin{center} below we will treat $I \in \mathcal{I}_k$ and $I \in \mathcal{I}_{k'}$ as two different elements. \end{center}
\vskip5pt

We now fix an $I \in \mathcal{I}_\ell$ which satisfies ($\ell=1$) or ($\ell \geq 2$) (or both), and define $S(I)$. First, we include $I=:\tilde{I}(0)$ in $S(I)$. Next, suppose that $m \geq 0$ and that $\tilde{I}(m) \in \mathcal{I}_{\ell+m}$ has been included in $S(I)$. 

If $\ell+m=L$, or if $\tilde{I}(m) \in \mathcal{J}(a1)$, then we stop the process.
Otherwise $\tilde{I}(m)$ has at least one child.  

\begin{itemize}
\item[(i)] If $\tilde{I}(m) \in \mathcal{J}(1) \cup \mathcal{J}(a2) \cup \mathcal{J}(a3)$, we include the only child $\tilde{I}(m+1)$ of $\tilde{I}(m)$ in $S(I)$. 
\item[(ii)] If $\tilde{I}(m)=[s_0,t_0] \in \mathcal{J}(b)$, then we include the child $\tilde{I}(m+1)$ of $\tilde{I}(m)$ with the same left endpoint $s_0$ in $S(I)$.  
\end{itemize} 

The process stops after $0 \leq n \leq L-\ell$ steps, and we let 
\begin{equation} \label{aunos} 
S(I)=\{I=\tilde{I}(0),\tilde{I}(1),\ldots,\tilde{I}(n)\}, \quad 0\leq n \leq L-\ell.    
\end{equation}
The collection $S(I)$ contains $I$ and one  grandchild of $I$ from each ``generation'' following $I$, continuing either until the last generation $L$, or until the chosen grandchild is in $\mathcal{J}(a1)$. 

\begin{lemma}\label{endptlemma}
Every segment $J \in \mathcal{J}(a) \setminus \{[0,1]\}$ belongs to exactly one $S(I)$.  
\end{lemma}

\begin{proof}
Fix $J \in \mathcal{J}_{k'}(a)$, $k' \geq 1$. If $k'=1$, then $J \in S(J)$. Otherwise, let $0 \leq k \leq k'-1$ be the largest integer so that 
\begin{enumerate} 
\item the grandparent $I' \in \mathcal{I}_{k}$ of $J$ belongs to $\mathcal{J}(b)$, and 
\item the left endpoint of $I'$ is different from the left endpoint of $I \in \mathcal{I}_{k+1}$, where $I$ is $J$ if $k=k'-1$, and the grandparent of $J$ in $\mathcal{I}_{k+1}$ otherwise.  
\end{enumerate}

If such a $k$ does not exist, or if $k=0$, then $J \in S(I)$, where $I$ is the grandparent of $J$ which is a child of $[0,1]$. 

If $k \geq 1$, then $J$ is $I$ or has a  grandparent $I$, which is a child of the segment $I' \in \mathcal{J}(b)$ above with a different left endpoint. Then $J \in S(I)$. 
We conclude that $J$ belongs to some $S(I)$ in all of the above cases. 

To prove uniqueness, assume that $J \in S(I)$, $I \in \mathcal{I}_\ell$, and notice that no element of $S(I)$ other than $I$ is of type ($\ell=1$) or $(\ell \geq 2)$ above. On the other hand, the construction of $S(I)$ shows that no grandparent $I'' \in \mathcal{I}_{k''}$ of $J$ satisfies $J \in S(I'')$ when $k'' <\ell$. Since $J$ can belong to $S(I'')$ only when $I''$ is $J$ or a grandparent of $J$, we conclude the uniqueness of the $I$ for which $J \in S(I)$. 
\end{proof}

We fix $S(I)=\{I=\tilde{I}(0),\tilde{I}(1),\ldots,\tilde{I}(n)\}$ as in \eqref{aunos}. Our next goal is to estimate the ``loss'' incurred by removing subsegments from the segments  
$$
\tilde{I}(m) \in S(I)\cap \big(\mathcal{J}(a1) \cup \mathcal{J}(a3)\big) 
$$ 
by analyzing the properties of the left endpoints of the segments $\tilde{I}(m)$ as $m$ increases. We denote by 
$0 \leq m_1 < m_2 < \cdots < m_\omega \leq n$ the indices 
$m_\psi$ for which 
\begin{equation}
\label{Mjummu}
\tilde{I}(m_\psi) \in S(I) \cap \big(\mathcal{J}_{\ell+m_\psi}(a1) \cup \mathcal{J}_{\ell+m_\psi}(a3)\big),  
\end{equation} 
i.e., for which $\tilde{I}(m_\psi)$ has no children or has exactly one child whose left endpoint is different from the endpoint of $\tilde{I}(m_\psi)$. Notice that 
$$
\tilde{I}(m_\psi) \in \mathcal{J}_{\ell+m_\psi}(a1) \quad 
\text{can only happen when } \psi=\omega \quad \text{and } m_{\psi}=n. 
$$
We assume that there is at least one index $m_\psi$ that satisfies \eqref{Mjummu}. We denote 
$$
\tilde{I}(m)=[c_m,d_m]. 
$$ 
We will apply the following properties of the left endpoints $c_m$. 

\begin{lemma} \label{sonke} 
Suppose that $0 \leq m \leq n$ and $\tilde{I}(m) \in \mathcal{J}(a1) \cup \mathcal{J}(a3)$. Then 
\begin{equation}
\label{sonke1}
\gamma(c_m) \subset \mathbb{D}(b_{\ell+m+1},r_{\ell+m+1}). 
\end{equation}
Moreover, if $\tilde{I}(m) \in \mathcal{J}(a3)$, 
then $\gamma(c_{m+1}) \in B_{\ell+m+1}$ and 
\begin{equation} 
\label{sonke2} 
\gamma(c_{m+1}) \subset \mathbb{D}(a_{\ell+m+1},4 \tau r_{\ell+m+1}).  
\end{equation}
\end{lemma}
\begin{proof} 
The first claim follows from \eqref{berb-a1} and \eqref{berb-a3}. The last claims follow from the definition of $\mathcal{J}(a3)$ and 
\eqref{karkaamo}. 
\end{proof}

Next, we consider a case where the left endpoints of the segments $\tilde{I}(m)$ do not change when $m$ increases.  

\begin{lemma} \label{ooti}
Let $0 \leq m' < m \leq n$ so that $\tilde{I}(m'') \notin \mathcal{J}(a3)$ for every $m' \leq m'' \leq m-1$. Then $c_{m'}=c_m$. 
\end{lemma}
\begin{proof}
This follows directly from the construction of $S(I)$, since the left endpoint of $\tilde{I}(m''+1)$ is different from the left endpoint of $\tilde{I}(m'')$ only if $\tilde{I}(m'') \in \mathcal{J}(a3)$. 
\end{proof}

We will apply the following lemmas to show that the radii $r_{\ell+m}$ associated to the indices $m$ of the segments $\tilde{I}(m)$ can be controlled by a geometric sum. 

\begin{lemma} \label{Szeerone} 
Assume the following conditions for $0 \leq m' < m \leq n$: 
\begin{enumerate} 
\item $\tilde{I}(m') \in \mathcal{J}(a3)$ and $\tilde{I}(m) \in \mathcal{J}(a1) \cup \mathcal{J}(a3)$,  
\item If $m'<m-1$, then $\tilde{I}(m'') \notin \mathcal{J}(a3)$ for every $m'<m''<m$. 
\end{enumerate}
Then 
\begin{equation} \label{ifis4}
r_{\ell+m+1} \leq \tau r_{\ell+m'+1}. 
\end{equation}
\end{lemma}

\begin{proof} 
We claim that 
\begin{equation} 
\label{ifis3} 
\mathbb{D}(a_{\ell+m'+1},2r_{\ell+m'+1}) \cap \mathbb{D}(a_{\ell+m+1},2r_{\ell+m+1}) \neq \emptyset. 
\end{equation} 
By Condition (1) and \eqref{sonke2}, we have 
\begin{equation} 
\label{ifis1} 
\gamma(c_{m'+1}) \subset \mathbb{D}(a_{\ell+m'+1},4 \tau r_{\ell+m'+1}) \subset \mathbb{D}(a_{\ell+m'+1},2r_{\ell+m'+1}), 
\end{equation} 
where the last inclusion holds since $\tau < 1000^{-1}$. 

By Lemma \ref{ooti} and Condition (2), we have $c_{m'+1}=c_m$. Moreover, Condition (1) and \eqref{sonke1} yield 
\begin{equation} 
\label{ifis2}
\gamma(c_{m'+1})=\gamma(c_{m}) \subset 
\mathbb{D}(b_{\ell+m+1},r_{\ell+m+1}) \subset 
\mathbb{D}(a_{\ell+m+1},2r_{\ell+m+1}), 
\end{equation}
where the last inclusion holds by \eqref{indika}. Now \eqref{ifis3} follows by combining \eqref{ifis1} and \eqref{ifis2}. Moreover, \eqref{ifis4} follows by combining \eqref{ifis3} and Lemma \ref{inter}. 
\end{proof}

The next lemma can be applied to the first terms of the sequence $S(I)$, assuming that they are ``good''. 

\begin{lemma} \label{Tzeerone}
Suppose that $m_1 \geq 1$. Then $c_{m_1}=c_0$. Moreover, if $I=\tilde{I}(0)$ is of type ($\ell \geq 2$) above and 
\begin{equation} \label{shrimp} 
\diam(\gamma(c_0)) < \tau r_\ell, 
\end{equation}
then 
\begin{equation} 
\label{maritu}
r_{\ell+m_1+1} \leq \tau r_\ell. 
\end{equation}
\end{lemma}
\begin{proof}
The first claim follows from Lemma \ref{ooti}. For the second claim, we recall from Condition ($\ell \geq 2$) that $c_0$ is different from the left endpoint of the parent $J \in \mathcal{J}(b)$ of $I$. Thus, the construction of $\mathcal{J}(b)$ yields 
$$ 
\gamma(c_0) \cap \overline{\mathbb{D}}(b_\ell,r_\ell) \neq \emptyset.  
$$ 
Combining with \eqref{shrimp}, we conclude that 
$$
\gamma(c_0)\subset 
\mathbb{D}(b_\ell,(1+\tau r_\ell)) \subset \mathbb{D}(a_\ell,2r_\ell),  
$$ 
where the last inclusion follows from \eqref{indika}. 
On the other hand, $\tilde{I}(m_1) \in \mathcal{J}(a1) \cup 
\mathcal{J}(a3)$ by the definition of $m_1$, and thus \eqref{sonke1} yields 
$$
\gamma(c_0)=\gamma(c_{m_1}) \subset \mathbb{D}(b_{\ell+m_1+1},r_{\ell+m_1+1}) \subset \mathbb{D}(a_{\ell+m_1+1},2r_{\ell+m_1+1}), 
$$  
where the last inclusion again follows from \eqref{indika}. Combining the inclusions, we have 
$$
\gamma(c_0)=\gamma(c_{m_1}) \subset 
\mathbb{D}(a_\ell,2r_\ell) \cap \mathbb{D}(a_{\ell+m_1+1},2r_{\ell+m_1+1}). 
$$ 
Now \eqref{maritu} follows from Lemma \ref{inter}. 
\end{proof}

\subsection{Proof of Proposition \ref{goodpathsprop}: Completion of the proof}
Recall that our goal is to prove \eqref{Feet}, i.e., 
\begin{eqnarray} \nonumber  
\sum_{k=1}^L (\card\mathcal{J}_{k-1}(a)) \cdot r_k &\leq& \frac{1}{20}  + \frac{4}{\tau}\sum_{p \in G(\gamma)} \diam(p)\\ \label{Feet2} &+& 12 \tau \sum_{k=1}^L (\card\mathcal{J}_{k-1}(b2))\cdot r_k, 
\end{eqnarray}

the final missing piece in the proof of Proposition \ref{goodpathsprop}. As we noticed in Section \ref{cssection}, we may assume that $L \geq 2$. We now combine the estimates obtained in Section \ref{cssection} to control the ``loss'' incurred by removing subsegments from the segments 
$$
\tilde{I}(m_\psi) \in S(I) \cap \big(\mathcal{J}_{\ell+m_\psi}(a1) \cup \mathcal{J}_{\ell+m_\psi}(a3)\big), \quad 0 \leq m_1 < m_2 < \cdots < m_\omega \leq n,  
$$ 
defined in \eqref{Mjummu}. Recall that $I=[c_0,d_0] \in \mathcal{I}_\ell$ for some $1 \leq \ell \leq L$, 
$$
S(I)=\{I=\tilde{I}(0),\tilde{I}(1),\ldots,\tilde{I}(n)\}, \quad 0\leq n \leq L-\ell, 
$$
and $\tilde{I}(m)=[c_m,d_m]$.  

\begin{lemma} \label{crois} If we denote $d(\ell):=\diam(\gamma(c_{0}))$, then 
\begin{eqnarray} \label{sampa}
\sum_{\psi=1}^\omega r_{\ell+m_\psi+1} \leq 
\left\{
\begin{array}{ll}
\frac{1}{200} & \text{if } \ell=1, \\ 
2 \tau r_\ell & \text{if } \ell \geq 2 \text{ and } d(\ell) < \tau r_\ell, \\ 
2 \tau^{-1} d(\ell) & \text{if } \ell \geq 2 \text{ and } d(\ell) \geq \tau r_\ell. 
\end{array}
\right. 
\end{eqnarray} 
\end{lemma}

\begin{proof}
We recall that $\tilde{I}(m_\psi)$ can belong to $\mathcal{J}(a1)$ only when $\psi=\omega$. Thus, by Lemma \ref{Szeerone} we have 
\begin{equation} 
\label{hupas}
r_{\ell+m_{\psi+1}+1} \leq \tau r_{\ell+m_{\psi}+1}
\quad \text{for every } 1 \leq \psi \leq \omega-1.  
\end{equation} 
Iterating \eqref{hupas} and recalling that $0<\tau < \frac{1}{1000}$ yields 
\begin{equation} \label{hupas2} 
\sum_{\psi=1}^\omega r_{\ell+m_\psi+1} \leq \sum_{\psi=1}^\omega \tau^{\psi-1}r_{\ell+m_1+1} 
\leq 2 r_{\ell+m_1+1}. 
\end{equation} 
 
 We are now ready to prove \eqref{sampa}. Suppose first that $\ell=1$. By \eqref{triviii},  
$$
r_{\ell+m_1+1} \leq r_1 \leq \tau^2 < \frac{1}{1000000}\, ,  
$$
which together with \eqref{hupas2} proves the first part of  \eqref{sampa}. 

Suppose next that $\ell \geq 2$ and $d(\ell)=\diam(\gamma(c_{0})) < \tau r_\ell$. Then Lemma \ref{Tzeerone} shows that 
$$
r_{\ell+m_1+1} \leq \tau r_\ell, 
$$
which combined with \eqref{hupas2} gives the second part of \eqref{sampa}. 

Finally, suppose that $\ell \geq 2$ and $d(\ell) \geq \tau r_\ell$. Then 
$$
r_{\ell+m_1+1} \leq r_\ell \leq \tau^{-1} d(\ell), 
$$ 
which combined with \eqref{hupas2} gives \eqref{sampa}. 
\end{proof}

\begin{proof}[Proof of Proposition \ref{goodpathsprop}]
 Recall that the proposition follows once we have proved \eqref{Feet2}. We first bound the sum 
\begin{equation} \label{creet} 
\sum_{k=1}^L (\card(\mathcal{J}_{k-1}(a1)\cup \mathcal{J}_{k-1}(a3))) \cdot r_k
\end{equation} 
from above. By Lemma \ref{endptlemma}, each 
\begin{equation}
\tilde{I} \in \mathcal{J}(a1)\cup \mathcal{J}(a3)) \setminus \{[0,1]\}=:\mathcal{K} 
\end{equation}
belongs to $S(I)$ for exactly one segment $I$. 

Suppose that $I \in \mathcal{I}_\ell$ and 
$\tilde{I} \in S(I) \cap \mathcal{K}$. We denote $\tilde{I} \in \mathcal{K}(\ell=1)$ if $\ell=1$, i.e., if $I$ is a child of $[0,1]$. If $\ell \geq 2$, then the parent $J$ of $I$ belongs to $\mathcal{J}(b1)$ or to $\mathcal{J}(b2)$, according to the type ($\ell \geq 2$) in Section \ref{cssection}. We denote 
$\tilde{I} \in \mathcal{K}(b1)$ if $J \in \mathcal{J}(b1)$, 
and $\tilde{I} \in \mathcal{K}(b2)$ if $J \in \mathcal{J}(b2)$. Then $\mathcal{K}$ is the union 
\begin{equation} \label{tormet}
\mathcal{K}=\mathcal{K}(\ell=1) \cup \mathcal{K}(b1) \cup \mathcal{K}(b2).      
\end{equation} 
Moreover, each set on the right is the union of disjoint sets 
of the form $S(I)\cap \mathcal{K}$. 
We denote $\mathcal{K}_m(\ell=1)=\mathcal{K}(\ell=1) \cap \mathcal{I}_m$, and apply similar notation for the other three sets in \eqref{tormet}. Then 
\begin{eqnarray} \label{creet1} 
& & \sum_{k=1}^L (\card(\mathcal{J}_{k-1}(a1)\cup \mathcal{J}_{k-1}(a3))) \cdot r_k = \sum_{k=2}^L (\card(\mathcal{K}_{k-1}(\ell=1)))\cdot r_k  \\ \nonumber
&+& \sum_{k=2}^L (\card(\mathcal{K}_{k-1}(b1)))\cdot r_k + \sum_{k=2}^L (\card(\mathcal{K}_{k-1}(b2)))\cdot r_k + \chi r_1, 
\end{eqnarray} 
where $\chi=1$ if $[0,1] \in \mathcal{J}(a3)$ and $\chi=0$ otherwise (recall that $[0,1] \notin \mathcal{J}(a1)$). 

We now estimate the part of \eqref{creet1} involving $\mathcal{K}(\ell=1)$. We notice that since $[0,1] \in \mathcal{I}_0$ has at most four children, there are at most four segments $I \in \mathcal{I}_1$. For each such $I$, Lemma \ref{crois} shows that 
$$
\sum_{k=2}^L (\card(\mathcal{K}_{k-1}(\ell=1)\cap S(I)))\cdot r_k \leq \frac{1}{200}.   
$$ 
Summing over $I \in \mathcal{I}_1$ then gives 
\begin{equation} 
\label{summi1} 
\sum_{k=2}^L (\card(\mathcal{K}_{k-1}(\ell=1)))\cdot r_k \leq \frac{1}{50}. 
\end{equation}

We next estimate the sum in \eqref{creet1} involving $\mathcal{K}(b1)$. Suppose that
$S(I)\cap \mathcal{K} \subset \mathcal{K}(b1)$. Then the parent $J=[s_0,t_0]$ of $I$ belongs to $\mathcal{J}_{\ell-1}(b1)$ for some $2 \leq \ell \leq L$. By the constructions of $\mathcal{J}(b1)$ in Section \ref{goodpathssec} and type ($\ell\geq 2$) 
in Section \ref{cssection}, the segment $J=[s_0,t_0]$ has two children, namely $[c'_{0},d'_0]$ and $[c_0,d_0]=I$. Moreover, by Lemma \ref{b1lemma} we have 
\begin{equation} \label{sunnun0}
\diam(\gamma(t))\geq \tau r_\ell \quad \text{and} \quad  \gamma(t) 
\in \bigcup_{j=1}^\ell G_j \quad \text{for } t=d'_0 \text{ or } 
t=c_0. 
\end{equation}
Thus, combining the last two cases of \eqref{sampa} in Lemma \ref{crois} shows that 
\begin{equation} \label{sunnun} 
\sum_{k=2}^L (\card(\mathcal{K}_{k-1}(b1)\cap S(I)))\cdot r_k \leq \frac{2}{\tau}\diam(\gamma(t)) 
\end{equation}
for $t=d'_0$ or $t=c_0$ (or both). 

Since neither of such points $t$ is an endpoint of $J$, an element $p \in G(\gamma)$ can appear as a $\gamma(t)$ in \eqref{sunnun0} for at most one segment $J \in \mathcal{J}(b1)$. Thus, summing \eqref{sunnun} over all the segments $I$ for which $S(I) \cap \mathcal{K}\subset \mathcal{K}(b1)$, we have 
\begin{equation} \label{summi2} 
\sum_{k=2}^L (\card(\mathcal{K}_{k-1}(b1)))\cdot r_k \leq  \frac{2}{\tau} \sum_{p \in G(\gamma)} \diam(p). 
\end{equation} 

Finally, we estimate the sum in \eqref{creet1} involving $\mathcal{K}(b2)$. Suppose that $S(I) \cap \mathcal{K} \subset \mathcal{K}(b2)$. Then the parent $J$ of $I=[c_0,d_0]$ belongs to $\mathcal{J}_{\ell-1}(b2)$ for some $2 \leq \ell \leq L$. By the construction of $\mathcal{J}(b2)$ it follows that 
$\diam(\gamma(c_0)) < \tau r_\ell$. Thus, by Lemma \ref{crois} 
we have 
\begin{equation} \label{sunnun1} 
\sum_{k=2}^L (\card(\mathcal{K}_{k-1}(b2)\cap S(I)))\cdot r_k \leq 2 \tau r_\ell.   
\end{equation}

Every $J \in \mathcal{J}(b2)$ has at most three children $I$ for which there is a sequence $S(I)$, i.e., which are of the type ($\ell \geq 2$) in Section \ref{cssection}. Therefore, summing \eqref{sunnun1} over all the segments $I$ for which $S(I) \cap \mathcal{K} \subset \mathcal{K}(b2)$, we have 
\begin{equation} \label{summi3} 
\sum_{k=2}^L (\card(\mathcal{K}_{k-1}(b2)))\cdot r_k \leq  6 \tau \sum_{k=1}^L (\card\mathcal{J}_{k-1}(b2))\cdot r_k.
\end{equation} 

We now combine \eqref{creet1}, \eqref{summi1}, \eqref{summi2} and \eqref{summi3}, and recall that $r_1 \leq \frac{1}{1000000}$ by \eqref{triviii}, to obtain 
\begin{eqnarray*} 
\sum_{k=1}^L (\card(\mathcal{J}_{k-1}(a1)\cup \mathcal{J}_{k-1}(a3))) \cdot r_k \leq \frac{1}{40}  + \frac{2}{\tau}\sum_{p \in G(\gamma)} \diam(p)\\  + 6 \tau \sum_{k=1}^L (\card\mathcal{J}_{k-1}(b2))\cdot r_k. 
\end{eqnarray*} 
We can replace $\mathcal{J}(a3)$ by $\mathcal{J}(a2)$, and run the argument above, but now considering the right endpoints instead of the left endpoints, to show that 
\begin{eqnarray*} 
\sum_{k=1}^L (\card(\mathcal{J}_{k-1}(a1)\cup \mathcal{J}_{k-1}(a2))) \cdot r_k \leq \frac{1}{40}  + \frac{2}{\tau}\sum_{p \in G(\gamma)} \diam(p)\\  + 6 \tau \sum_{k=1}^L (\card\mathcal{J}_{k-1}(b2))\cdot r_k. 
\end{eqnarray*} 
Since $\mathcal{J}(a)=\mathcal{J}(a1) \cup \mathcal{J}(a2) \cup \mathcal{J}(a3)$, combining the two estimates gives \eqref{Feet2}. The proofs of Proposition \ref{goodpathsprop} and Theorem \ref{tripodkoebe} are complete. 
\end{proof} 

\begin{remark} \label{unifyremark}
Theorem \ref{tripodkoebe} admits the following generalization: Let $\Omega \subset \hatc$ be a domain containing $\infty$, and suppose that $\mathcal{C}_N(\Omega)= T \cup F$, where 
\begin{enumerate} 
\item the elements of $T$ satisfy Conditions (i) and (ii) of Theorem \ref{tripodkoebe}, and 
\item there is a $C>0$ so that every $p \in F$ is $C$-fat. 
\end{enumerate} 
Then the conclusions of Theorem \ref{tripodkoebe} hold. 

We give a brief outline of how the proof of Theorem \ref{tripodkoebe} should be modified to establish such a generalization. First, by the arguments in Section \ref{sec:Simon}, it suffices to prove Theorem \ref{thm:mainestimate} for finitely connected domains whose complementary components are elements of $T \cup F$. 

To prove Theorem \ref{thm:mainestimate}, we notice that Proposition \ref{blockingprop} remains true for every $p \in T$. As in Section \ref{gblsection} and using the fatness condition, we see that it suffices to consider paths $\gamma$ that do not pass through the ``large'' elements of $T \cup F$. Let 
$T_S$ and $F_S$ be the sets of the ``small'' components, whose diameters are smaller than $\tau$, in $T$ and $F$, respectively. 

We partition $T_S$ into good and bad sets as in Section \ref{gblsection}, so that $T_S= G \cup B$. Moreover, we define the function $\rho$ as in \eqref{testfunction} on  $\Omega$ and $T_S$, and complete the definition by setting $\rho(p)=(M+1)\diam(p)$ if $p \in F_S$. 

The fatness condition and the arguments after \eqref{testfunction} guarantee that the energy $\int_{\Omega} \rho^2 \, dA + \sum_{p \in \mathcal{C}(\Omega)} \rho(p)^2$ is uniformly bounded from above. Therefore, it suffices to prove that $\rho$ is admissible in the current setting. 

To prove the admissibility or $\rho$, we continue to apply the cases 
$$
(1),(a1),(a2),(a3),(b1),(b2)
$$ 
in Section \ref{goodpathssec}, and run the proof as above. In case (b1) and Lemma \ref{b1lemma} the $\alpha(s_1)$ and $\alpha(t_1)$ may now be elements of $F_S$. As a result, the key estimate \eqref{kiikos} in Proposition \ref{goodpathsprop} holds when the last sum is over the set $G(\gamma) \cup F(\gamma)$; here 
$$
F(\gamma)=\{p: p \in F \cap |\gamma|\}.   
$$ 
The estimate is strong enough to guarantee the admissibility of $\rho$, so Theorem \ref{thm:mainestimate} is true in our setting. 
\end{remark}

\section{Proofs of modulus estimates on circle domains, Proposition \ref{circlemodulus}} \label{sec:cmodulus}
We fix a $\bar{p} \in \mathcal{C}(\Omega)$, a Jordan curve $J \subset \Omega$, and points $b,a$ as in the proposition. Let $j \geq 1$ if $\bar{p} \in \mathcal{C}_P(\Omega)$ and $j \geq \ell$ if $\bar{p}=p_\ell \in \mathcal{C}_N(\Omega)$. Then $\hat{f}_j(\bar{p})$ is a generalized disk or a point in $\hatc$. In the following proof it is convenient to replace the normalization \eqref{normali}, which was applied to guarantee the injectivity of limit map $f$, with a new normalization.   

Namely, since transboundary modulus and generalized disks are invariant under M\"obius transformations, we lose no generality by replacing  the sequence $(f_j)_j$ with 
$(h \circ f_j)_j$, where $h$ is any M\"obius transformation. 
Therefore, by choosing $h$ suitably we may assume that 
\begin{equation} \label{foto}
\hat{f}_j(\bar{p}) \cup f_j(J) \subset \mathbb{D}(0,1),  \quad \infty \in D_j, \text{ and }  f_j(J) \text{ separates }\hat{f}_j(\bar{p}) \text{ and }\infty. 
\end{equation}

We start with the first estimate in Proposition \ref{circlemodulus}, i.e., 
\begin{equation}\label{afirst} 
\limsup_{j \to \infty} \Mod \hat{f}_j(\Gamma_j) \geq \limsup_{j \to \infty} \varphi_{\bar p}(\dist(f_j(b),\hat{f}_j(\bar{p}))). 
\end{equation} 

We denote $\dist(f_j(b),\hat{f}_j(\bar{p}))$ by $\delta_j$. Let $w_0$ be the point in $\hat{f}_j(\bar{p})$ closest to $f_j(b)$. After a rotation about the origin, 
$f_j(b)=\delta_j i+w_0$. Since $f_j(J)$ separates $\hat{f}_j(\bar{p})$ and $\infty$, it follows that every line $L_s=\{t+si+w_0: t \in \mathbb{R}\}$, $0<s<\delta_j$, has a subsegment $I_s \subset U \subset \mathbb{D}(0,1)$ so that 
$\pi_{D_j}(I_s) \in \hat{f}_j(\Gamma_j)$. Here $U$ is the 
bounded component of $\mathbb C \setminus f_j(J)$. 

Recall that we are under the assumption that $\Omega_j$ has no point boundary components, so that $\mathcal{C}(D_j)$ consists of disks. Let $\rho$ be admissible for $\hat{f}_j(\Gamma_j)$. Then 
\begin{equation} 
\label{iiala}
1 \leq \int_{I_s \cap D_j} \rho \, ds + 
\sum_{q \in \mathcal{C}^s(D_j)} \rho(q) \quad \text{for all } 0<s<\delta_j,  
\end{equation}
where $\mathcal{C}^s(D_j)=\{q \in \mathcal{C}(D_j): \, I_s \cap q \neq \emptyset\}.$ Combining \eqref{iiala} with Fubini's theorem yields 
\begin{equation} \label{kkori}
\delta_j \leq \int_{D_j \cap U} \rho \, dA + \sum_{q \in \mathcal{C}_U} \diam(q) \rho(q),  
\end{equation} 
where $\mathcal{C}_U=\{q \in \mathcal{C}(D_j): q \subset U\}$. 
By the Cauchy-Schwarz inequality (since $U \subset \mathbb{D}(0,1)$) we have 
\begin{eqnarray*}  
 \int_{D_j \cap U} \rho \, dA &\leq& \operatorname{Area}(U)^{1/2} \Big( \int_{D_j} \rho^2 \, dA \Big)^{1/2} \leq \pi^{1/2} \Big( \int_{D_j} \rho^2 \, dA \Big)^{1/2},
 \end{eqnarray*} 
and
\begin{eqnarray*} 
 \sum_{q \in \mathcal{C}_U} \diam(q) \rho(q) &\leq&  \Big(\sum_{q \in \mathcal{C}_U} \diam(q)^2 \Big)^{1/2} \Big(\sum_{q \in \mathcal{C}(D_j)} \rho(q)^2 \Big)^{1/2} \\ &\leq& 
2\Big(\sum_{q \in \mathcal{C}(D_j)} \rho(q)^2 \Big)^{1/2}. 
\end{eqnarray*} 
Combining with \eqref{kkori}, we obtain
\begin{align*}
    \delta_j &\le \pi^{1/2} \Big( \int_{D_j} \rho^2 \, dA \Big)^{1/2} + 2\Big(\sum_{q \in \mathcal{C}(D_j)} \rho(q)^2 \Big)^{1/2} \\
    &\le (\pi+4)^{1/2} \Big(\int_{D_j} \rho^2 \, dA +  \sum_{q \in \mathcal{C}(D_j)} \rho(q)^2\Big)^{1/2}.  
\end{align*}

Taking infimum with respect to all admissible functions shows that 
$$
\Mod \hat{f}_j(\Gamma_j) \ge \frac{\delta_j^2}{\pi+4}. 
$$
In particular, \eqref{afirst} holds. 

We now consider the second estimate in Proposition \ref{circlemodulus}, i.e.,  
\begin{equation} \label{second}
\text{if } \diam(\hat{f}(\bar{p}))=0 \text{ then } \lim_{j \to \infty} \Mod \hat{f}_j (\Lambda_j)  \to \infty.  
\end{equation} 
Notice that the first claim in \eqref{nina2} does not depend on \eqref{second}, so by \eqref{afirst} and the proof given in Section \ref{sec:Simon} we already know that $\hat{f}(p_\ell)=q_\ell$ for every $\ell=1,2,\ldots$. 
In particular, the generalized disks $q_\ell$ are pairwise disjoint. By our assumption and normalization \eqref{foto} we have 
$$\hat{f}(\bar{p})=\{w_0\} \quad \text{where }w_0 \in \mathbb{C}. 
$$ 
We need a technical lemma. 
 
\begin{lemma} \label{twotimer}
For every $R>0$ there are $j_R \in \mathbb{N}$ and $0<r<R$ such that if $j \geq j_R$ and if $q \in \mathcal{C}(D_j)$ satisfies $q \cap \mathbb{S}(w_0,R) \neq \emptyset$, then 
$q \cap \mathbb{S}(w_0,r) = \emptyset$. 
\end{lemma}
\begin{proof} 
Suppose towards a contradiction that there are $R>0$, a subsequence $(f_{j_k})_k$ of $(f_j)_j$, and components $p^*(k) \in \mathcal{C}(\Omega_{j_k})$, so that 
\begin{itemize} 
\item[(a)] every $q^*(k):=\hat{f}_{j_k}(p^*(k))$ intersects $\mathbb{S}(w_0,R)$, and 
\item[(b)] $\dist(q^*(k),w_0) \to 0$ as $k \to \infty$. 
\end{itemize} 
Notice that none of the sets $p^*(k)$ are $\bar p$. Taking a subsequence if necessary, we may assume that the sets $p^*(k)$ converge to some $p^* \in \mathcal{C}(\Omega)$ in the Hausdorff sense. We denote $q^*:=\hat{f}(p^*)$. Then, by Carath\'eodory's kernel convergence theorem and the convergence of $(f_{j_k})_k$ to $f$, for every $\epsilon>0$ there are $\delta>0$ and a $k(\epsilon) \in \mathbb{N}$ such that if $k \geq k(\epsilon)$ then the neighborhoods of the sets $p^*$ and $q^*$ satisfy 
$$
\hat{f}_{j_k}(N_\delta(p^*)) \subset N_\epsilon(q^*). 
$$
On the other hand, by the Hausdorff convergence there is 
a $k'(\delta)$ such that if $k \geq k'(\delta)$ then 
$p^*(k)\subset N_\delta(p^*)$. 

We conclude that $q^*(k)\subset N_\epsilon(q^*)$ 
for sufficiently large indices $k$. Letting $\epsilon \to 0$ and applying (b), we see that $q^*$ must be $\{w_0\}$. In particular, the radii of the disks $q^*(k)$ converge to zero. But by (a) and (b), these radii cannot converge to zero. We arrive at a contradiction.  
\end{proof}

We construct a sequence of annuli as follows (compare to the proof of \eqref{upperbound}): 
Let $r_1$ be the number satisfying $\dist(f(J),w_0)=10r_1$. Since $f_j \to f$ locally uniformly in $\Omega$, we may assume that 
$\dist(f_j(J),w_0)\geq 5r_1$ for all $j$. Assuming $r_1,\ldots, r_{n-1}$ are defined, we apply Lemma \ref{twotimer} with $R=r_{n-1}/2$ to obtain an index $j'_n \in \mathbb{N}$ and a radius $0<r_n'<r_{n-1}/2$ such that if $j \geq j'_n$ and if $q \in \mathcal{C}(D_j)$ intersects $\mathbb{S}(w_0,r_{n-1}/2)$ then $q$ does not intersect $\mathbb{S}(w_0,r_n')$. We then let 
$$
r_n=\frac{r_n'}{10}.  
$$
Here it is important that the radii $r_n$ do not depend on $j$. We let 
$$
\mathbb{A}_n=\mathbb{D}(w_0,4r_n)\setminus \overline{\mathbb{D}}(w_0,r_n/2), \quad n=1,2,\ldots.  
$$
Now fix an $M \geq 1$, a Jordan curve $J' \subset \Omega$ surrounding $\bar{p}$, and an index $j^*_M \in \mathbb{N}$, so that 
$$
f_j(J') \subset \mathbb{D}(w_0,r_M/10) \quad \text{for all } j \geq j^*_M; 
$$
such choices are possible since $\hat{f}(\bar{p})=\{w_0\}$. By our choices of the radii $r_n$ we also have 
\begin{equation} \label{nicks} 
 \pi_{D_j}(\mathbb{A}_{n}) \cap \pi_{D_j}(\mathbb{A}_{m}) =\emptyset \quad \text{for all }  1 \leq n < m \leq M \text{ and } j \geq j_M,   
\end{equation} 
where $j_M=\max\{j^*_M,j'_1,j'_2,\ldots, j'_M\}$. 

Let $1 \leq n \leq M$. Given $r_n/2<t<4r_n$, we denote by $\tilde{\gamma}_t$ the circle $\mathbb{S}(w_0,t)$ parameterized by arclength, 
$\gamma_t=\pi_{D_j} \circ \tilde{\gamma}_t$, and 
$$
\Phi_j(n)=\{\gamma_t: \,  r_n/2<t<4r_n \}. 
$$ 
Then $\Phi_j(n) \subset \hat{f}_j(\Lambda_j)$. We next prove a lower bound for $\Mod(\Phi_j(n))$. Let $\rho$ be admissible for $\Phi_j(n)$ and 
$r_n/2<t<4r_n$. Then 
\begin{equation} \label{reka}
1 \leq \int_{\mathbb{S}(w_0,t)\cap D_j} \rho \, ds + 
\sum_{ q \cap |\gamma_t| \neq \emptyset} \rho(q).  
\end{equation} 
We divide both sides of \eqref{reka} by $t$ and integrate (in $t$) from $r_n/2$ to $4r_n$ to conclude, upon using Fubini's theorem, that 
\begin{equation} \label{reka2}
\log 8 \leq \int_{\mathbb{A}_n \cap D_j} \frac{\rho(z)}{|z|} \, dA(z) + \frac{2}{r_n}\sum_{q \cap \mathbb{A}_n \neq \emptyset} \min\{\diam(q),4r_n\}\rho(q). 
\end{equation} 
We apply the Cauchy-Schwarz inequality to estimate the integral on the right: 
\begin{eqnarray*}
\int_{\mathbb{A}_n \cap D_j} \frac{\rho(z)}{|z|} \, dA(z) &\leq& \Big(\int_{\mathbb{A}_n \cap D_j} \frac{dA(z)}{|z|^2}\Big)^{1/2} \Big(\int_{\mathbb{A}_n \cap D_j} \rho(z)^2 \, dA(z) \Big)^{1/2} \\  
 &\leq & (2\pi \log 8)^{1/2} \Big(\int_{\mathbb{A}_n \cap D_j} \rho(z)^2 \, dA(z) \Big)^{1/2}. 
\end{eqnarray*} 
To estimate the sum in \eqref{reka2}, we denote 
\begin{eqnarray*}
\mathcal{Q}_L &=& \{q \in \mathcal{C}(D_j): \, q \cap \mathbb{A}_n \neq \emptyset, \, \diam(q) \geq r_n \}, \\ 
\mathcal{Q}_S &=& \{q \in \mathcal{C}(D_j): \, q \cap \mathbb{A}_n \neq \emptyset, \, \diam(q) < r_n \}. 
\end{eqnarray*} 
Then 
\begin{equation} \label{miju}
\card \mathcal{Q}_L \leq 100 \quad \text{and} \quad q \subset \mathbb{D}(w_0,5r_n) \text{ for all } q \in \mathcal{Q}_S, 
\end{equation} 
and
\begin{eqnarray*} 
\frac{2}{r_n}\sum_{q \cap \mathbb{A}_n \neq \emptyset} \min\{\diam(q),4r_n\}\rho(q) \leq 
8 \sum_{q \in \mathcal{Q}_L} \rho(q) +\frac{2}{r_n} \sum_{q \in \mathcal{Q}_S} \diam(q) \rho(q). 
\end{eqnarray*} 
By the Cauchy-Schwarz inequality and \eqref{miju} we have 
\begin{eqnarray*} 
 \sum_{q \in \mathcal{Q}_L} \rho(q) \leq 10 \Big(\sum_{q \in \mathcal{Q}_L} \rho(q)^2 \Big)^{1/2}. 
\end{eqnarray*} 
Since the disks $q$ are pairwise disjoint, the Cauchy-Schwarz inequality and \eqref{miju} also yield  
\begin{eqnarray*} 
\sum_{q \in \mathcal{Q}_S} \diam(q)\rho(q) &\leq& \Big( \sum_{q \in \mathcal{Q}_S} \diam(q)^{2}\Big)^{1/2} \Big(\sum_{q \in \mathcal{Q}_S} \rho(q)^2 \Big)^{1/2}  \\ 
&\leq& \Big(\frac{4\operatorname{Area}(\mathbb{D}(w_0,5r_n))}{\pi} \Big)^{1/2} \Big(\sum_{q \in \mathcal{Q}_S} \rho(q)^2 \Big)^{1/2} \\
&\leq& 10r_n \Big(\sum_{q \in \mathcal{Q}_S} \rho(q)^2 \Big)^{1/2}. 
\end{eqnarray*} 

Combining the estimates with \eqref{reka2},  applying the Cauchy-Schwarz inequality again, and taking infimum over all $\rho$ shows that 
\begin{equation} \label{reka4}
\Mod(\Phi_j(n)) \geq \frac{1}{6400} \quad \text{for all } j \geq j_M \text{ and } 1 \leq n \leq M. 
\end{equation} 
Since $\Phi_j(n) \subset \hat{f}_j(\Lambda_j)$, combining \eqref{nicks} and \eqref{reka4} shows that 
$$
\Mod(\hat{f}_j(\Lambda_j)) \geq   \frac{1}{6400} M  
$$
for all $j \geq j_M$. Letting $M \to \infty$ proves \eqref{second}. The proof of Proposition \ref{circlemodulus} is complete. 

\section{Necessity of the packing condition in Theorem \ref{tripodkoebe}} \label{sec:examples}
In this section, we illustrate the need for both Conditions (i) and (ii) in Theorem~\ref{tripodkoebe}. We start with the need for Packing Condition (ii). 
\begin{proposition}\label{tripodexample}
There exists a countably connected domain $\Omega \subset \hatc$ which contains $\infty$ and satisfies  Quasitripod Condition (i) (but not Packing Condition (ii)) in Theorem \ref{tripodkoebe} so that $\{0\} \in \mathcal{C}(\Omega)$ and $\diam(\hat{f}(\{0\}))>0$ for 
every conformal homeomorphism $f\colon\Omega \to D$ onto a circle domain $D$.
\end{proposition}
\begin{proof}
We construct the desired domain $\Omega$ by describing the elements of $\mathcal{C}(\Omega)$. First, $\{0\}$ is the only element of $\mathcal{C}_P(\Omega)$. The collection $\mathcal{C}_N(\Omega)$ is parameterized as follows: Given a $k \in \mathbb{N}$, we denote by $W_k$ the collection of finite words $w=w_1\cdots w_k$, where $w_j \in \{0,1\}$ for every $1 \leq j \leq k$. 
Moreover, let $W_0=\{\emptyset\}$ and $W=\bigcup_{k=0}^\infty W_k$. We then have 
$$
\mathcal{C}_N(\Omega)=\{p_w: \, w \in W\}.   
$$
The words $w \in W$ are ordered so that $0 <1 <00<01<10<11<000 \ldots$. We denote the order of $w$ by $\ell(w)$. Notice that $\ell(w)$ is not word length, but rather the order of $w$ in this enumeration. 

We set $\ell(\emptyset)=0$, and let $p_\emptyset$ be the segment $[\frac{1}{2},1]$. If $\ell(w) \geq 1$, each $p_w$ is the union of radial segments $I_w,J_w$ and subarcs 
$S_w,T_w$ of circles centered at the origin.  If $w=\bar w w_k$, where $\ell(\bar w) \geq 0$ and $w_k \in \{0,1\}$, then $I_w$ is a segment of length 
$2^{-\ell-2}-\epsilon_\ell$, $\ell=\ell(\bar w)$, in the annulus 
$$
\mathbb{A}_\ell=\overline{\mathbb{D}}(0,2^{-\ell}) \setminus 
\mathbb{D}(0,2^{-\ell-1}),  
$$
where $\epsilon_\ell>0$ is a small number. The segments $I_{\bar w 0}$ and 
$I_{\bar w 1}$ are subsets of the same half-line starting at the origin.  

The arc $S_w$ is attached to the middle of $I_w$ and has length $\frac{1}{24}$ times the length of the full circle. The arc $T_w$ is roughly a half-circle, attached to an end of $I_w$, and lies in $\mathbb{S}(0,3\cdot2^{-\ell-2})$ if $w_k=0$ and in $\mathbb{S}(0,2^{-\ell-1})$ if $w_k=1$. The segment $J_w$ is attached to an end of $T_w$. The other end of $J_w$ lies at the circle $\mathbb{S}(0,2^{-\ell(w)-1})$. Recall that $\ell(w)$ is the ordering of $w$ and not the word length, so $\ell(w)$ tends to be much larger than $\ell$. The distance between $I_w$ and $J_{\bar w}$ is less than $\epsilon_\ell$. 

Figure \ref{fig1} shows the segments $I_{00},J_{00}$, arcs $S_{00},T_{00}$, components  $p_{00}$, $p_{01}$, $p_{10}$, $p_{11}$, $p_{000}$, $p_{001}$, and parts of components $p_0$, $p_1$. The sequence $(\epsilon_\ell)_\ell$ can be chosen so that the elements $p_w$ have the following properties: 
\begin{enumerate} 
\item For every $w \in W$ there is a $c_w>0$ so that $p_w$ is the image of $c_wT_0=\{c_wz: \, z \in T_0\}$ under a $10^6$-biLipschitz map. In particular, each $p_w$ is a $10^{12}$-quasitripod. 
\item For every $\epsilon>0$ there is a $k_\epsilon \geq 1$ so that if the word length $|w|=k \geq k_\epsilon$ 
then $p_w \subset \mathbb{D}(0,\epsilon)$. 
\item For every $w=\bar{w}w_{k}$, $w_k \in \{0,1\}$, there is a family $\Gamma_w$ of paths connecting $p_{\bar{w}}$ and $p_w$ 
in $\Omega$ so that $\mod(\Gamma_w) > 4^{k}$. More precisely, $\Gamma_w$ consists of short subarcs of circles in $\mathbb{A}_\ell$ centered at the origin. 
\end{enumerate}

\begin{figure}
\begin{tikzpicture}[scale=.8]
\draw[dashed] (0,0) circle (4cm);\draw[dashed] (0,0) circle (8cm);
\draw[dashed] (0,0) circle (2cm);
\draw[dashed] (0,0) circle (1cm);
\draw[dashed] (0,0) circle (.5cm);
\draw[dashed] (0,0) circle (.25cm);
\draw[dashed] (0,0) circle (.125cm);

\draw[thick] (canvas polar cs:angle=0,radius=4cm) -- 
(canvas polar cs:angle=0,radius=12cm);
\draw[thick] (canvas polar cs:angle=2,radius=2cm) -- 
(canvas polar cs:angle=2,radius=12cm);

\draw[thick] (canvas polar cs:angle=-2,radius=4cm) 
-- (canvas polar cs:angle=-2,radius=5.9cm);
\draw[thick] (canvas polar cs:angle=-2,radius=6cm) 
-- (canvas polar cs:angle=-2,radius=7.9cm);
\draw[thick] (canvas polar cs:angle=0,radius=3cm) 
-- (canvas polar cs:angle=0,radius=3.9cm);
\draw[thick] (canvas polar cs:angle=0,radius=2cm) 
-- (canvas polar cs:angle=0,radius=2.9cm);

\draw[thick] (canvas polar cs:angle=-2,radius=7cm) arc (-2:-17:7cm);
\draw[thick] (canvas polar cs:angle=-2,radius=5cm) arc (-2:-17:5cm);
\draw[thick] (canvas polar cs:angle=0,radius=3.5cm) arc (0:-15:3.5cm);
\draw[thick] (canvas polar cs:angle=0,radius=2.5cm) arc (0:-15:2.5cm); 

\draw[thick] (canvas polar cs:angle=-2,radius=6cm) arc (-2:-180:6cm);
\draw[thick] (canvas polar cs:angle=-2,radius=4cm) arc (-2:-176:4cm);
\draw[thick] (canvas polar cs:angle=0,radius=3cm) arc (0:-172:3cm);
\draw[thick] (canvas polar cs:angle=0,radius=2cm) arc (0:-168:2cm);

\draw[] (canvas polar cs:angle=-180,radius=6cm) 
-- (canvas polar cs:angle=-180,radius=1cm);
\draw[] (canvas polar cs:angle=-176,radius=4cm) 
-- (canvas polar cs:angle=-176,radius=.5cm);
\draw[] (canvas polar cs:angle=-172,radius=3cm) 
-- (canvas polar cs:angle=-172,radius=.25cm);
\draw[] (canvas polar cs:angle=-168,radius=2cm) 
-- (canvas polar cs:angle=-168,radius=.125cm);

\draw[thick] (canvas polar cs:angle=-182,radius=1.95cm) 
-- (canvas polar cs:angle=-182,radius=1.5cm);
\draw[thick] (canvas polar cs:angle=-182,radius=1.45cm) 
-- (canvas polar cs:angle=-182,radius=1cm);

\draw[thick] (canvas polar cs:angle=-182,radius=1.25cm) arc (-182:-197:1.25cm);
\draw[thick] (canvas polar cs:angle=-182,radius=1.75cm) arc (-182:-197:1.75cm);

\draw[thick] (canvas polar cs:angle=-182,radius=1cm) arc (-182:-352:1cm);
\draw[thick] (canvas polar cs:angle=-182,radius=1.5cm) arc (-182:-356:1.5cm);

\draw[] (canvas polar cs:angle=-352,radius=1cm) 
-- (canvas polar cs:angle=-352,radius=.0625cm);
\draw[] (canvas polar cs:angle=-356,radius=1.5cm) 
-- (canvas polar cs:angle=-356,radius=.03125cm);


\node at (9,-.3) {$p_{0}$};
\node at (6,.5) {$p_{1}$};

\node at (0,-6.3) {$p_{00}$};
\node at (0,-4.3) {$p_{01}$};
\node at (0,-3.3) {$p_{10}$};
\node at (0,-2.3) {$p_{11}$};
\node at (0,1.7) {$p_{000}$};
\node at (0,.7) {$p_{001}$};

\node at (7.5,-.5) {$I_{00}$};
\node at (6.5,-2.5) {$S_{00}$};
\node at (4,-5) {$T_{00}$};
\node at (-5,.3) {$J_{00}$};

\end{tikzpicture}
\caption{Some complementary components of the domain $\Omega$ constructed in the proof of Proposition~\ref{tripodexample}.}
\label{fig1} 
\end{figure}

Since $\Omega$ is countably connected, the He-Schramm theorem \cite{HS:93} guarantees the existence of a conformal homeomorphism $f\colon\Omega \to D$ 
onto a circle domain $D$. Moreover, $f$ is unique up to postcomposition by a M\"obius transformation. To show that 
$\hat{f}(\{0\}) \in \mathcal{C}_N(D)$, we denote by $\Gamma$ the family of paths in $\hat{\Omega}$ joining $p_\emptyset$ and $\{0\}$.

Towards a contradiction, assume that $\hat{f}(\{0\})$ is a point-component. Then we have $\Mod(\hat{f}(\Gamma)) =0$, which can be proved by applying 
\cite[Theorem 6.1(2)]{Sch:95} to a sequence of annuli (or by modifying the proof of \eqref{upperbound} in the special case of circle domains). 
Since the transboundary modulus is conformally invariant (Lemma \ref{modinvariance}), 
the desired contradiction will follow once we prove that 
\begin{equation}
\label{tilalaskenta}
\Mod(\Gamma)>0. 
\end{equation} 

We denote by $W_\infty$ the collection of infinite words $w_1w_2\cdots$, where $w_j \in \{0,1\}$. We equip $W_\infty$ with the unique probability measure $\mu$ satisfying $\mu(A_{w})=2^{-k}$ for all $k \geq 1$ and $w \in W_k$. Here  
$$
A_w=\{w_\infty \in W_\infty: \, w_\infty=ww_{k+1}w_{k+2}\cdots \}. 
$$
Let $\rho:\hat{\Omega} \to [0,\infty]$ be an arbitrary Borel function satisfying   
\begin{equation} 
\label{rohma}
\int_{\Omega} \rho^2 \, dA + \sum_{w \in W} \rho(w)^2=1. 
\end{equation} 
We will find a $v_\infty=v_1v_2\cdots \in W_\infty$ so that 
\begin{equation} \label{sahmo} 
\sum_{k=1}^\infty \rho(p_{\bar{v}_k}) \leq 1.  
\end{equation} 
Here $\bar{v}_k=v_{1}v_2\cdots v_k$. We first notice that 
\begin{eqnarray*} 
\int_{W_\infty} \sum_{k=1}^\infty \rho(p_{\bar{w}_k}) \, d\mu(w_\infty) 
= \sum_{k=1}^\infty \sum_{w \in W_k} \mu(A_w)\rho(p_w)
= \sum_{k=1}^\infty 2^{-k} \sum_{w \in W_k} \rho(p_w)=:S. 
\end{eqnarray*} 
The Cauchy-Schwarz inequality yields (notice that $\card W_k=2^{k}$) 
\begin{eqnarray*}
S \leq \sum_{k=1}^\infty 2^{-k/2} \Big(\sum_{w \in W_k} \rho(p_w)^2 \Big)^{1/2} 
\leq \Big( \sum_{k=1}^\infty 2^{-k} \Big)^{1/2} \Big( \sum_{w \in W} \rho(p_w)^2 \Big)^{1/2} \leq 1, 
\end{eqnarray*} 
where the last inequality follows from \eqref{rohma}. Combining the estimates shows that there indeed 
exists a $v_\infty=v_1v_2\cdots \in W_\infty$ satisfying \eqref{sahmo}. 

Recall that for each $\bar{v}_k=v_1v_2\cdots v_k$, $k=1,2,\ldots$, there is a family $\Gamma_{\bar{v}_k}$ 
of paths connecting $p_{\bar{v}_{k-1}}$ and $p_{\bar{v}_{k}}$ in $\Omega$ so that 
$\Mod(\Gamma_{\bar{v}_k}) > 4^k$. Now \eqref{rohma} implies that for every $k$ there is an 
$\gamma_k \in \Gamma_{\bar{v}_k}$ so that 
\begin{equation} \label{ubles}
\int_{\gamma_k} \rho \, ds < 2^{-k}. 
\end{equation}
Indeed, otherwise $2^k \rho$ would be admissible for $\Gamma_{\bar{v}_k}$ and thus 
$\Mod(\Gamma_{\bar{v}_k}) \leq 4^k$ by \eqref{rohma}, which is a contradiction. 
Concatenating the paths $\pi_\Omega \circ \gamma_k$, $k=1,2,\ldots$, yields a path $\gamma \in \Gamma$ 
so that $|\gamma| \cap \mathcal{C}(\Omega)$ only contains $\{0\}$, $p_\emptyset$, and the elements $p_{\bar{v}_k}$, $k=1,2,\ldots$. Combining \eqref{sahmo} and \eqref{ubles} gives 
\begin{equation} \label{anko}
\int_{\gamma \cap \Omega} \rho \, ds + \sum_{k=1}^\infty \rho(p_{\bar{v}_k}) \leq 2. 
\end{equation} 
We have proved that for every $\rho$ satisfying \eqref{rohma} there is an $\gamma \in \Gamma$ which satisfies \eqref{anko}. Lemma \ref{modlemma} now shows that \eqref{tilalaskenta} holds. We conclude that $\Omega$ has all the desired properties.
\end{proof}

\begin{remark}
It is also possible to construct a countably connected domain $\Omega \subset \hatc$ which satisfies Packing Condition (ii) (but not Quasitripod Condition (i)) in Theorem \ref{tripodkoebe}, so that $\{0\} \in \mathcal{C}(\Omega)$ and $\diam(\hat{f}(\{0\}))>0$ for every conformal homeomorphism $f\colon\Omega \to D$ onto a circle domain $D$. Namely, one can modify the \emph{dyadic slit domain} construction of Hakobyan and Li \cite{HakLi23} on unions of squares to get a domain $\tilde{\Omega} \subset \hatc$ whose complementary components are $\infty$ and vertical segments  contained in an infinite strip, and define $\Omega$ as the image of $\tilde{\Omega}$ under the inversion $z \mapsto z^{-1}$. 

The construction can be carried out so that $\Omega$ satisfies Condition (ii) in Theorem \ref{tripodkoebe}, and so that if $J \subset \Omega$ is a Jordan curve then  $\mod(\Gamma)>0$ for the family $\Gamma$ of paths in $\hat{\Omega}$ that connect $\pi_\Omega(J)$ and $\pi_\Omega(\{0\})$; see \cite[Lemma 7.1]{HakLi23}. On the other hand, if $f\colon \Omega \to D$ is a conformal homeomorphism onto a (countably connected) circle domain such that $\hat{f}(\pi_\Omega(\{0\}))$ is a point-component, then it follows as in the proof of Proposition \ref{tripodexample} that $\mod \hat{f}(\Gamma)=0$. This contradicts Lemma \ref{modinvariance}, and so $\hat{f}({\pi_\Omega(\{0\}}))$ must be a disk. We leave the details to the interested reader. 
\end{remark} 

\begin{example} \label{elltwoexample} 
We describe a domain $\Omega \subset \hatc$ for which the assumptions of Theorem \ref{tripodkoebe} are satisfied but 
\begin{equation} \label{slay}
\sum_{p \in \mathcal{C}_N(\Omega)} \diam(p)^2=\infty. 
\end{equation}

We apply a well-known construction of a Cantor set $K \subset \mathbb{C}$ with positive area; see e.g. \cite[Proof of Theorem 4.10]{HenKos14}. Start by dividing the square $[-1,1]^2$ into four congruent subsquares $Q'_j=Q(z_j,\frac{1}{2})$, $j \in \{1,2,3,4\}$, with disjoint interiors; here  $Q(z,r)$ is the closed square with center $z$, side length $2r$, and sides parallel to the coordinate axes. 
Set $Q_j=Q(z_j,\frac{1}{2}-\epsilon_1)$. 

Proceeding by induction, suppose that we have constructed $4^k$ disjoint squares $Q_v=Q(z_v,r_k)$, $v \in \{1,2,3,4\}^k$. 
We divide each $Q_v$ into four congruent subsquares $Q'_{vj}=Q(z_{vj},\frac{r_k}{2})$, $j \in \{1,2,3,4\}$, and set $Q_{vj}=Q(z_{vj},\frac{r_k}{2}-\epsilon_{k+1})$. The Cantor set $K$ is 
\begin{equation} \label{ahmu}
K= \bigcap_{k=1}^\infty \bigcup_{v \in \{1,2,3,4\}^k} Q_v. 
\end{equation}
The parameters $\epsilon_k>0$ can be chosen so that $\operatorname{Area}(K)>0$.

\begin{figure}[h]
\begin{center}
\begin{tikzpicture}[scale=0.1]
  \newcommand{\drawplus}[2]{%
    \draw[thick] ($#1 + (-#2/2,0)$) -- ($#1 + (#2/2,0)$);
    \draw[thick] ($#1 + (0,-#2/2)$) -- ($#1 + (0,#2/2)$);
  }
   \newcommand{\drawplusthick}[2]{%
    \draw[very thick] ($#1 + (-#2/2,0)$) -- ($#1 + (#2/2,0)$);
    \draw[very thick] ($#1 + (0,-#2/2)$) -- ($#1 + (0,#2/2)$);
  }
   \newcommand{\drawplusdotted}[2]{%
    \draw[thick, densely dotted] ($#1 + (-#2/2,0)$) -- ($#1 + (#2/2,0)$);
    \draw[thick, densely dotted] ($#1 + (0,-#2/2)$) -- ($#1 + (0,#2/2)$);
  }
  \drawplusthick{(0,0)}{64}
  \foreach \x in {-16,16} {
    \drawplusthick{(\x,16)}{30}
    }
      \foreach \x in {-16,16} {
    \drawplusthick{(\x,-16)}{30}
    }

\foreach \y in {-24,-8,8,24} {
  \foreach \x in {-24,-8,8,+24} {
    \drawplus{(\x,\y)}{14}
    }
    }
28
\foreach \y in {-28,-20,-12,-4,4,12,20,28} {
  \foreach \x in {-28,-20,-12,-4,4,12,20,28} {
    \drawplusdotted{(\x,\y)}{5.2}
    }
    }
\end{tikzpicture}
\end{center}
\caption{Example of a domain whose non-trivial complementary components (the ``plus'' signs) satisfy the conditions of Theorem~\ref{tripodkoebe} and their diameters are not $\ell^2$-summable.}
\label{plus-signs}
\end{figure}

In order to define $\Omega$, we fix $v$ in \eqref{ahmu} and let $p_v$ be the largest ``plus sign'' in $Q_v$. That is, $p_v$ is the union of two horizontal and two vertical segments of equal length, each connecting $z_v$ to one side of $Q_v$. The continua $p_v$ are pairwise disjoint and do not intersect $K$. 

Let $\Omega \subset \hatc$ be the domain whose complement is the union of $K$ and all the continua $p_v$. 
The union of any three of the four segments which define the plus sign is the image of a quasisymmetric (in fact, bi-Lipschitz) map from the standard tripod. Thus, $\Omega$ satisfies Condition (i) in Theorem \ref{tripodkoebe}. In order to prove \eqref{slay}, we 
notice that $\diam(p_v)^2 = \operatorname{Area}(Q_v)$ for every $v$. Therefore, we have 
$$
\sum_{v \in \{1,2,3,4\}^k} \diam(p_v)^2 \geq 
\sum_{v \in \{1,2,3,4\}^k} {\operatorname{Area}(Q_v)} \geq {\operatorname{Area}(K)}>0 
$$
for every $k \in \mathbb{N}$. Summing over $k$ implies \eqref{slay}. 

It remains to prove the Packing Condition (ii) in Theorem \ref{tripodkoebe}. We fix a point $z_0 \in \mathbb{C}$ and a radius $r>0$. By covering $\mathbb{D}(z_0,r)$ with a collection of disks $\mathbb{D}_\alpha:=\mathbb{D}(z_\alpha,\frac{r}{100})$ so that the disks $\mathbb{D}(z_\alpha,\frac{r}{200})$ are pairwise disjoint, it suffices to prove the cardinality bound  
\begin{equation} 
\label{sanka}
\card \mathcal{C}_\alpha=\card\{p_v \in \mathcal{C}_N(\Omega): \, 
\diam(p_v) \geq r, \, p_v \cap \mathbb{D}_\alpha \neq \emptyset \} 
\leq 6  
\end{equation} 
for every $\alpha$. Estimate \eqref{sanka} follows from two geometric properties which are straightforward to verify: 
\begin{enumerate} 
\item At most two disjoint squares $Q_v$ satisfy $p_v \in \mathcal{C}_\alpha$. 
\item If $p_v \in \mathcal{C}_\alpha$, then there are at most two squares $Q_{v'} \subsetneq Q_v$ such that  $p_{v'}\in \mathcal{C}_\alpha$. 
\end{enumerate}
We have established the desired properties of $\Omega$. 
\end{example}

\section{Cospread domains, Proof of Proposition \ref{cospreadcon}} \label{sec:propoproof}
To start the proof of Proposition \ref{cospreadcon}, we notice that the definition of cospread domains already contains Quasitripod Condition (i) in Theorem \ref{tripodkoebe}. We state the remaining claims of Proposition \ref{cospreadcon} as the following two propositions. 

\begin{proposition}
    \label{pack1}
    Let $\Omega \subset \hatc$ be an $H$-cospread domain. There is an $N$ which depends only on $H$ so that for every $z_0 \in \mathbb{C}$ and $r>0$, 
    \begin{equation}
    \label{eq1}
 \card \{p \in \mathcal{C}_N(\Omega): \,\diam(p) \geq r, \,  p \cap \mathbb{D}(z_0,r) \neq \emptyset \} \leq N.
    \end{equation}
\end{proposition} 

\begin{proposition} \label{Quasiinv} 
Let $\Omega \subset \hatc$ be an $H$-cospread domain and $\phi \colon \hatc \to \hatc$ an $\alpha$-quasi-M\"obius map. Then $\phi(\Omega)$ is $H'$-cospread, where $H'$ depends only on $H$ and $\alpha$. 
\end{proposition}

We recall the definition of quasi-M\"obius maps. The \emph{cross-ratio} of distinct points 
$z_1,z_2,z_3,z_4 \in \hatc$ is 
$\displaystyle
[z_1,z_2,z_3,z_4]:= \frac{q(z_1,z_2)q(z_3,z_4)}{q(z_1,z_3)q(z_2,z_4)}, 
$
where $q$ is the chordal distance defined by 
\begin{eqnarray*}
q(z,w)=\frac{|z-w|}{\sqrt{1+|z|^2}\sqrt{1+|w|^2}} \quad \text{and} \quad 
q(z,\infty)=\frac{1}{\sqrt{1+|z|^2}}, \quad \text{ $z,w \in \mathbb{C}$}.
\end{eqnarray*}
A homeomorphism $\phi \colon \hatc \to \hatc$ is \emph{quasi-Möbius} if there is a homeomorphism $\alpha\colon [0,\infty) \to [0,\infty)$ so that 
\begin{equation}\label{Hool}
[\phi(z_1),\phi(z_2),\phi(z_3),\phi(z_4)] \leq \alpha([z_1,z_2,z_3,z_4]) 
\end{equation} 
for all distinct $z_1,z_2,z_3,z_4 \in \hatc$. To emphasize the role of $\alpha$, we use the term \emph{$\alpha$-quasi-M\"obius}. Notice that Möbius transformations are quasi-Möbius maps with $\alpha(t)=t$. 

Recall that a homeomorphism $\phi \colon E \to F$ between subsets of $\mathbb{C}$ is (strongly) $\eta$-\emph{quasisymmetric} if there is a homeomorphism $\eta\colon [0,\infty) \to [0,\infty)$ so that for all $ z_1,z_2,z_3 \in E$ { satisfying } $|z_2-z_1| \leq t|z_3-z_1|$, $0<t<\infty$, we have 
$$
|\phi(z_2)-\phi(z_1)| \leq \eta(t)|\phi(z_3)-\phi(z_1)|.
$$ 
It follows from the definitions that compositions and inverses $\phi$ of quasi-M\"obius (resp., quasisymmetric) maps $\phi_1$ and/or $\phi_2$ are quasi-M\"obius (resp., quasisymmetric). Moreover, the control function $\alpha$ (resp., $\eta$) of $\phi$ depends only on the control functions of $\phi_1$ and/or $\phi_2$ (see \cite[Prop. 10.6]{Hei:01}). If $E \subset \mathbb{C}$ is connected, then by V\"ais\"al\"a's theorem (which was already applied in the proof of Proposition \ref{blockingprop}), weakly $H$-qua\-si\-sym\-met\-ric maps $\phi \colon E\to F$ are $\eta$-quasisymmetric with $\eta$ depending only on $H$. 
\subsection{Proof of the packing condition, Proposition \ref{pack1}} 

    We fix a point $z_0 \in \mathbb{C}$ and a radius $r>0$, and denote  
    $$
    \mathcal{P}:=\{ p \in \mathcal{C}_N(\Omega): \diam(p) \geq r, \, p \cap \mathbb{D}(z_0,r) \neq \emptyset \} \, .
    $$
    Given $p \in \mathcal{P}$, we choose a point $z_p \in p \cap \mathbb{D}(z_0,r)$. Since $r \leq \diam(p)$ and $p$ is $H$-spread, there is an $H$-quasitripod $T_p \subset p \cap \mathbb{D}(z_p,r)$ with $\diam(T_p) \geq r/H$. Clearly $T_p \subset \mathbb{D}(z_0,2r)$. Since the quasitripods $T_p$ are pairwise disjoint,  claim \eqref{eq1} is an immediate consequence of the next lemma.

\begin{lemma}
\label{th2}
     Let $M,H \geq 1$ and suppose that $\mathcal{T}$ is a collection of pairwise disjoint $H$-quasitripods $T \subset \mathbb{D}(z_0,Mr)$ satisfying $\diam(T) \geq r$. Then $\card \mathcal{T} \leq N$, where $N$ depends only on $M$ and $H$. 
\end{lemma}
\begin{proof}
    
    Given $T \in \mathcal{T}$, recall that there is an $\eta$-quasisymmetric homeomorphism $\phi_T\colon T_0 \to T$; here $T_0$ is the standard tripod. We call $\phi_T(0)$ the \textit{center} $0_T$ of $T$ and the   components of $T \setminus 0_T$ \textit{the branches} of $T$. 

    We fix $0 < \delta < 1$ to be chosen later and cover $\mathbb{D}(z_0,Mr)$ with disks $D_1,\ldots, D_n$ of radius $\delta r$ so that $n \leq 100(M \delta^{-1})^2$. Given $1 \leq k \leq n$, we denote by $\mathcal{T}_k$ the collection of elements $T \in \mathcal{T}$ for which $0_T \in D_k$. 
    Since $\mathcal{T}=\bigcup_k \mathcal{T}_k$, the lemma follows if we can choose $\delta$ depending only on $H$ so that for some $N=N(H)$,  
    \begin{equation} \label{cardbound}
    \card \mathcal{T}_k \leq N \quad \text{for all } 1 \leq k \leq n.  
    \end{equation}
    
    Towards \eqref{cardbound}, a straightforward application of quasisymmetry shows that if $T \in \mathcal{T}_k$ and if $\delta$ is small enough, depending on $H$, then each of the branches $J_1(T),J_2(T),J_3(T)$ 
    of $T$ must leave $B_k=2D_k$, since $\diam(T) \geq r$. Here $2D_k$ is the disk with the center of $D_k$ and twice the radius. For $s \in \{1,2,3\}$, let $\alpha_s^T(t), 0\leq t \leq 1$, be a homeomorphic parameterization of $J_s(T)$ with $\alpha_s^T(0)=0_T$. We denote $a_s^T=\alpha_s^T(t_s)$, where 
$$
t_s:=\inf \{t: \ \alpha^T_s(t) \in \partial B_k\}. 
$$
The points $a^T_1,a^T_2,a^T_3$ partition $\partial B_k$ into subarcs $S_1(T),S_2(T),S_3(T)$. Another straightforward application of quasisymmetry shows that their lengths satisfy 
\begin{equation} \label{kaurism} 
\ell(S_s(T)) \geq \theta r \quad \text{for all } s \in \{1,2,3\}, 
\end{equation}
where $\theta >0$ depends only on $H$. 

We fix $S_T \in \{S_1(T),S_2(T),S_3(T)\}$ so that $\ell(S_T) \leq \ell(S_s(T))$ for $s \in \{1,2,3\}$. Fix a finite subcollection $\{T_1,T_2,\ldots,T_{L}\}$ of $\mathcal{T}_k$ so that 
$\ell(S_{T_1})\leq \ell(S_{T_2}) \leq \cdots \leq \ell(S_{T_L})$. We denote 
$S_{T_m}$ and $\ell(S_{T_m})$ by $S_m$ and $\ell_m$, respectively.
\begin{figure}[h]
\begin{center}
\begin{tikzpicture}[scale=.3]
\draw[dashed] (-10,0) circle (8cm);
\draw [very thick] plot [smooth ] coordinates {(-9,-2) (-15.7,-0.6) (-20,1)};
\draw [very thick] plot [smooth ] coordinates { (-9,-2) (-11, +2) (-15,+8)};
\draw [very thick] plot [smooth ] coordinates {(-9,-2) (-5,2) (-1,7)};
\draw [very thick] plot [smooth ] coordinates {(-10,-1) (-19,2) };
\draw [very thick] plot [smooth ] coordinates {(-10,-1) (-19, +4)};
\draw [very thick] plot [smooth ] coordinates {(-10,-1) (-18,8)};

\draw[dashed] (10,0) circle (8cm);
\draw [very thick] plot [smooth ] coordinates {(11,-2) (4.3,-0.6) (0,1)};
\draw [very thick] plot [smooth ] coordinates {(11,-2) (9, +2) (5,+8)};
\draw [very thick] plot [smooth ] coordinates {(11,-2) (15,2) (19,7)};

\draw [very thick] plot [smooth ] coordinates {(12,-2) (19,1)};
\draw [very thick] plot [smooth ] coordinates {(12,-2) (19,-2)};
\draw [very thick] plot [smooth ] coordinates {(12,-2) (15,-8)};

\filldraw node(p) at (-7,1.1) (above) {$T_2$};
\filldraw node(p) at (13,1.1) (above) {$T_2$};
\filldraw node(p) at (-15,2.5) (above) {$T_1$};
\filldraw node(p) at (15,-3) (below) {$T_1$};
\end{tikzpicture}
\end{center}
\caption{The possible relations between a pair of quasitripods with centers close to each other and large diameters.}
\label{packing-fig}
\end{figure}

Next, notice that there is an $s \in \{1,2,3\}$ so that $ a_{s'}^{T_1} \in S_s(T_2)$ for every $s'=1,2,3$. In particular, by our choice of the subarcs $S_T$ and the enumeration of the quasitripods $T_j$, either (Figure~\ref{packing-fig})
\begin{enumerate}
\item $S_1 \cap S_2 =\emptyset$ \quad (if $S_s(T_2)\neq S_2$), or 
\item $S_2$ contains $S_1$ and another subarc $S_{s'}(T_1)$ 
 \quad (if $S_s(T_2)= S_2$). 
\end{enumerate} 
Using \eqref{kaurism} we see that in both cases  $\ell(S_1 \cup S_2) \geq \theta r + \ell(S_1)$. An inductive argument shows that if $2 \leq m \leq L$ then there are $1 \leq m' \leq m$ and 
$s \in \{1,2,3\}$ so that 
\begin{equation} \label{kaurismm}
S_s(T_{m'}) \subset S_m \setminus \Big(\bigcup_{l=1}^{m-1} S_l \Big) 
\quad \text{and so} \quad 
\ell\Big(\bigcup_{l=1}^{m} S_l \Big) \geq \theta r + \ell\Big(\bigcup_{l=1}^{m-1} S_l \Big). 
\end{equation} 
Applying \eqref{kaurismm} and induction yields
\begin{equation}\label{frigu}
L \theta r \leq \ell \Big(\bigcup_{l=1}^{L} S_l \Big) \leq \ell(\partial B_k)=4 \delta \pi r.   
\end{equation}
Since \eqref{frigu} holds for all finite subcollections of $\mathcal{T}_k$ and $\theta$ depends only on $H$, the desired bound \eqref{cardbound} holds. The proof is complete. 
\end{proof}

\subsection{Proof of quasi-M\"obius invariance, Proposition \ref{Quasiinv}}
We will apply the following well-known estimate, see e.g. \cite[Proposition 10.8]{Hei:01}. 

\begin{lemma} \label{qsslemma}
Let $\nu\colon \overline{\mathbb{D}}(z_0,r) \to \nu(\overline{\mathbb{D}}(z_0,r))$ be $\eta$-quasisymmetric and let $A \subset \mathbb{D}(z_0,r)$ satisfy 
$$
\diam(\nu(A)) \geq \delta \min_{z \in \mathbb{S}(z_0,r)} |\nu(z)-\nu(z_0)|. 
$$ 
Then $\diam(A) \geq \delta'r$, where $\delta'$ depends only on $\delta$ and $\eta$. 
\end{lemma}

Let $\Omega \subset \hatc$ be $H$-cospread and $\phi \colon \hatc \to \hatc$ an $\alpha$-quasi-M\"obius map. Let $\varphi \colon \hatc \to \hatc$ be a M\"obius transformation so that 
$g=\varphi\circ \phi$ fixes infinity. Testing quasi-M\"obius condition \eqref{Hool} with the quadruple $z_1,z_2,z_3,\infty$ shows that $g|_\mathbb{C}$ is $\alpha$-quasisymmetric. Therefore, since $\phi=\varphi^{-1}\circ g$ it suffices to prove the claim of Proposition~\ref{Quasiinv} for quasisymmetric maps and M\"obius transformations.

We fix $p \in \mathcal{C}_N(\phi(\Omega))$, $z_0 \in p \cap \mathbb{C}$, and $r\leq \diam(p)$. Our goal is to show that $p \cap \mathbb{D}(z_0,r)$ contains a quasitripod with diameter comparable to $r$, under the assumption that $\phi$ is a quasisymmetric map or a M\"obius transformation.

First, let $\phi$ be $\eta$-quasisymmetric and let $\ell=\min_{z \in \mathbb{S}(z_0,r)}|\nu(z)-\nu(z_0)|$, where $\nu=\phi^{-1}$. Since $\nu(p)$ is $H$-spread by assumption, there is an $H$-quasitripod $T \subset \mathbb{D}(\nu(z_0),\ell) \cap \nu(p)$ with 
$\diam(T) \geq \ell/H$. Then, since  compositions of quasisymmetric maps are quasisymmetric, $\phi(T) \subset p \cap \mathbb{D}(z_0,r)$ is an $H_1$-quasitripod, where $H_1$ depends only on $H$ and $\eta$. The inverse of a quasisymmetric map is a quasisymmetric map, and the control functions depend only on each other. Thus, Lemma \ref{qsslemma} shows that $\diam(\phi(T)) \geq r/H_2$, where $H_2$ depends only on $H$ and $\eta$. We conclude that $\phi(\Omega)$ is $(\max\{H_1,H_2\})$-cospread.   

We now show that $\phi(\Omega)$ is cospread when $\phi$ is a M\"obius transformation. If $\phi$ fixes infinity then the claim is obvious. It therefore suffices to prove the claim for the inversion $\phi(z)= z^{-1}$. The following lemma follows directly from the definition of quasisymmetry. 

\begin{lemma} \label{kissa}
Let $\phi(z)= z^{-1}$ and suppose that $s>0$ and $w_0 \in \mathbb{C}$ satisfy $|w_0| \geq 2s$. Then $\phi|_{\overline{\mathbb{D}}(w_0,s)}$ is $\eta$-quasisymmetric with $\eta(t)=3t$. 
\end{lemma}

Now, if the point $z_0 \in p \cap \mathbb{C}$ above satisfies $|z_0| \geq r/10$ then $\phi^{-1}=\phi$ is quasisymmetric on $\mathbb{D}(z_0,r/20)$ by Lemma \ref{kissa}. On the other hand, if $|z_0| \leq r/10$ then we choose any $w_0 \in p \cap 
\mathbb{S}(z_0,r/2)$ (such a $w_0$ exists since $\diam(p) \geq r$) and notice that $|w_0| \geq r/10$. Lemma \ref{kissa} then shows that $h$ is quasisymmetric on $\mathbb{D}(w_0,r/20)\subset \mathbb{D}(z_0,r)$. 

Let $k_0=z_0$ if $|z_0| \geq r/10$ and $k_0=w_0$ otherwise. Since $\phi^{-1}(p)$ is spread by assumption, applying quasisymmetry and Lemma \ref{qsslemma} as above shows that  
$$
p \cap \mathbb{D}(k_0,r/20) \subset 
p \cap \mathbb{D}(z_0,r),
$$
and $p$ contains an $H'$-quasitripod with diameter bounded from below by $r/H'$, where $H'$ depends only on $H$. It follows that $p$ is $H'$-spread. The proofs of Propositions \ref{Quasiinv} and \ref{cospreadcon} are complete. 
\bibliographystyle{alpha}
\bibliography{Bibliography}
\end{document}